\documentclass[11pt, letterpaper]{amsart}
\usepackage[margin=1in]{geometry}
\usepackage{microtype}
\usepackage[abbrev]{amsrefs}
\usepackage{xcolor}
\usepackage{mathtools}
\usepackage{amssymb}
\usepackage{mathrsfs}
\usepackage{dsfont}
\usepackage{esint}
\usepackage{amsthm}
\usepackage{longtable}
\usepackage{hyperref}
\hypersetup{
  colorlinks   = true, 
  urlcolor     = blue, 
  linkcolor    = blue, 
  citecolor   = red 
}
\usepackage[normalem]{ulem}

\theoremstyle{plain}
 \newtheorem{theorem}{Theorem}[section]
 \newtheorem*{nonum-theorem}{}

 \newtheorem{lemma}[theorem]{Lemma}
 \newtheorem{corollary}[theorem]{Corollary}
 \newtheorem{proposition}[theorem]{Proposition}

\newtheoremstyle{mystyle}%
    {}
    {}%
    {\itshape}%
    {}%
    {\bfseries}%
    {.}%
     { }%
    {\thmnote{#3}}
 \theoremstyle{mystyle}
\newtheorem{mythm}{}

\theoremstyle{definition}
 \newtheorem{definition}[theorem]{Definition}
 \newtheorem{example}[theorem]{Example}
 \newtheorem*{ack}{Acknowledgment}
\theoremstyle{remark}
 \newtheorem{remark}[theorem]{Remark}

\DeclareMathOperator{\mk}{M\MRkern K}
\newcommand{\MRkern}{%
  \mkern-5.8mu
  \mathchoice{}{}{\mkern0.2mu}{\mkern0.5mu}%
}

\newcommand{\scrmk}[1][\sigma]{\mathcal{M\MRkern K}^{#1}}

\def\R{\mathbb{R}}
\def\N{\mathbb{N}}

\DeclareMathOperator*{\esssup}{ess\ sup}
\DeclareMathOperator{\Id}{Id}
\DeclareMathOperator{\Isom}{Isom}
\DeclareMathOperator{\avg}{avg}

\DeclareMathOperator{\inj}{inj}
\DeclareMathOperator{\Vol}{Vol}
\DeclareMathOperator{\dist}{\mathrm{d}}
\numberwithin{equation}{section}
\newcommand{\barypower}{\kappa}
\AtBeginDocument{%
   \def\MR#1{}
}

\begin{document}
\title[]{Barycenters in Disintegrated optimal transport}
\author[]{Jun Kitagawa}
\address{Department of Mathematics, Michigan State University, East Lansing, MI 48824, USA}
\email{kitagawa@math.msu.edu}
\author[]{Asuka Takatsu}
\address{Graduate School of Mathematical Sciences, The University of Tokyo, Tokyo {153-8914}, Japan \&
 RIKEN Center for Advanced Intelligence Project (AIP), Tokyo {103-0027}, Japan.}
\email{asuka-takatsu@g.ecc.u-tokyo.ac.jp}
%
\keywords{optimal transport, duality, barycenters, fiber bundles, disintegration of measures}
\subjclass[2020]{
49Q22, 
49Q20, 	
28A50} 
\begin{abstract}
 We prove existence and duality on a wide class of metric spaces, and uniqueness results on any connected, complete Riemannian manifold, with or without boundary, for classical Monge--Kantorovich barycenters. In particular, this is the first and only uniqueness result with no restriction on the geometry of the manifold aside from connectedness and completeness. We obtain these via the corresponding results for barycenter problems associated to a new two-parameter family of metrics on probability measures on a general metric fiber bundle, 
called the \emph{disintegrated Monge--Kantorovich metrics} (previously introduced by the authors).
\\[-20pt]
\end{abstract}
\maketitle
\section{Introduction}
First, given a complete, separable metric space $(X,\dist_X)$, we write $\mathcal{P}(X)$ for the space of Borel probability measures on~$X$, and for $1\leq p<\infty$, we let $\mathcal{P}_p(X)$ denote the subset of measures in~$\mathcal{P}(X)$ with finite $p$th moment. Then, $\mk_p^{X}$ will denote the \emph{$p$-Monge--Kantorovich metric} (sometimes 
also called the \emph{$p$-Wasserstein metric}) on $\mathcal{P}_p(X)$, defined via optimal transport.
That is, for $\mu$, $\nu\in \mathcal{P}_p(X)$,~define
\begin{align*}
\begin{split}
 \Pi(\mu,\nu)\coloneqq &\{\gamma\in \mathcal{P}(X\times X)\mid 
 \gamma(A\times X)=\mu(A), \gamma(X\times A)=\nu(A) \ \text{for any Borel }A\subset X
 \},\\
 \mk_p^X(\mu, \nu)\coloneqq &
 \inf_{\gamma\in \Pi(\mu, \nu)}\left\|\dist_X\right\|_{L^{p}(\gamma)}
 =\inf_{\gamma\in \Pi(\mu, \nu)}\left(\int_{X\times X} \dist_X(x,y)^pd\gamma(x, y)\right)^{\frac1p}.
\end{split}
\end{align*}
By~\cite{Villani09}*{Theorem~4.1} the infimum above is always attained, such a minimizer is called \emph{a $p$-optimal coupling} of $\mu$ and $\nu$. We will also write for $K\geq 2$,
\begin{align*}
    \Lambda_K\coloneqq \left\{(\lambda_k)_{k=1}^K\in (0, 1)^K\Biggm| \sum_{k=1}^K\lambda_k=1\right\}.
\end{align*}
for the probability simplex. 
Finally, if $X$ is a locally compact Hausdorff space, a real valued function $\phi$ on $X$ \emph{vanishes at infinity} if
\[
\left\{ x\in X \bigm| \left|\phi(x)\right| \geq \varepsilon \right\}
\]
is compact for all $\varepsilon>0$, and we denote by $C_0(X)$ the space of continuous functions on $X$
vanishing at infinity, equipped with the supremum norm.

Then our first main result concerns barycenters with respect to these classical $\mk^X_p$ metrics and is as follows.
\begin{theorem}\label{thm: classical barycenters}
Fix $K\in \N$, $K\geq 2$, $(\lambda_k)_{k=1}^K\in \Lambda_K$, $1\leq p<\infty$.
Let $(Y,\dist_Y)$ be a complete, separable metric space and fix some subset $\{\mu_k\}_{k=1}^K\subset\mathcal{P}_p(Y)$.
\begin{enumerate}
\setlength{\leftskip}{-15pt}
    \item \label{thm: disint bary exist2}
    If $(Y, \dist_Y)$ satisfies the Heine--Borel property, for any $\barypower> 0$ there exists a minimizer of the function 
\begin{align}
\label{eqn: classical barycenter}
\nu \mapsto \sum_{k=1}^K \lambda_k \mk_p^Y\left(\mu_k,\nu\right)^\barypower \quad\text{in } \mathcal{P}_p(Y).
\end{align}
    
    \item\label{thm: classical bary duality} If $(Y, \dist_Y)$ is locally compact, 
    \begin{align*}
        \inf_{\nu\in \mathcal{P}_p(Y)}\sum_{k=1}^K \lambda_k \mk_p^Y\left(\mu_k,\nu\right)^p
        =\sup\left\{-\sum_{k=1}^K \int_Y\phi_k^{\lambda_k\dist_Y^p}d\mu_k\Biggm| \frac{\left| \phi_k\right|}{1+\dist_Y(y_0, \cdot)^p}\in C_0(Y),\ \sum_{k=1}^K\phi_k\equiv 0\right\}.
    \end{align*}
    \item\label{thm: classical bary uniqueness} If $p>1$, $\kappa=p$, and $Y$ is a complete, connected Riemannian manifold, possibly with boundary, and $\mu_k$ is absolutely continuous with respect to the Riemannian volume measure on $Y$ for some $1\leq k\leq K$, then there is a unique minimizer in $\mathcal{P}_p(Y)$ of the function~\eqref{eqn: classical barycenter}.
\end{enumerate}
\end{theorem}
As one immediate application, Theorem~\ref{thm: barycenters} yields results for classical $\mk_p^Y$-barycenters in a wide variety of spaces. 
The above result extend the duality result of Agueh and Carlier in~\cite{AguehCarlier11}*{Proposition~2.2} to any locally compact metric space, and the uniqueness result to \emph{all} complete, connected Riemannian manifolds with or without boundary, with \emph{no restriction on Riemannian geometry} (for example, regarding injectivity radius or curvature bounds). In particular, the uniqueness result is the first and only such result with no restriction on geometry. 

This theorem follows from Theorem~\ref{thm: barycenters} below, which are the results of Theorem~\ref{thm: classical barycenters}  in the setting of the \emph{disintegrated Monge--Kantorovich metrics} $\scrmk_{p, q}$. These are a two parameter family of metrics on spaces of probability measures on \emph{metric fiber bundles} introduced by the authors in \cite{KitagawaTakatsu25a} (see Definition~\ref{def: scrdist}). We note the corresponding uniqueness result Theorem~\ref{thm: barycenters}~\eqref{thm: disint bary unique} (hence also Theorem~\ref{thm: classical barycenters}~\eqref{thm: classical bary uniqueness}) is sharp in the restrictions $p>1$ and $q<\infty$, as demonstrated in Examples~\ref{ex: q=infinity not unique} and~\ref{ex: p=1 nonunique}. Additionally, the techniques used to prove the uniqueness result are of independent interest and likely to be of use in the analysis of other variational problems. In the following section we provide the necessary preliminary definitions, then give the statement of Theorem~\ref{thm: barycenters}.
\subsection{Notation and preliminary definitions}
We start by recalling the main setting of the paper, of \emph{metric fiber bundles}. 
Given a metric space $(X,d_X)$, we denote by $\Isom(X)$ the isometry group of $X$. Recall, an action by a subgroup $G$ of $\Isom (X)$ on $X$ is \emph{effective}
if 
$gx=x$ for all $x\in X$ implies $g$ is the identity element in~$G$. 
\begin{definition}\label{def: metric fiber bundle}
    A \emph{metric fiber bundle} is a triple of metric spaces $(E, \dist_E)$, $(\Omega, \dist_\Omega)$, and $(Y, \dist_Y)$, and a continuous, surjective map $\pi: E\to \Omega$ such that there exists an open cover $\{U_j\}_{j\in \mathcal{J}}$ of $\Omega$ and maps $\Xi_j: U_j\times Y\to \pi^{-1}(U_j)$ (\emph{local trivializations}) such that for each $j\in \mathcal{J}$,
    \begin{enumerate}\setlength{\leftskip}{-15pt} 
        \item  $\Xi_j$ is a homeomorphism between $U_j\times Y$ (endowed with the product metric) and $\pi^{-1}(U_j)$ with the restriction of $\dist_E$.
        \item 
         $\pi(\Xi_j(\omega, y))=\omega$ for all $(\omega, y)\in U_j\times Y$.
        \item Write $\Xi_{j, \omega}(y)\coloneqq \Xi_j(\omega, y)$ for $\omega\in U_j$. Then for $j'\in \mathcal{J}$ with $U_j\cap U_{j'}\neq \emptyset$, there is a subgroup $G$ of $\Isom(Y)$ (called a \emph{structure group}) acting on $Y$ effectively, and a
        map $g_{j}^{j'}:U_j\cap U_{j'}\to G$ (which is well-defined since $G$ is effective) such that 
\[
\Xi_{j', \omega}^{-1} (\Xi_{j, \omega}(y))=g_{j}^{j'}(\omega) y
\quad \text{for }(\omega,y)\in (U_j\cap U_{j'}) \times Y.
\]
\item For $\omega\in U_j$, the map $\Xi_{j, \omega}\colon Y \to \pi^{-1}(\{\omega\})$ is an isometry.
    \end{enumerate}
\end{definition}
The reader should keep in mind the case when $E=\Omega\times Y$ is a trivial bundle: $G$ is the trivial group, and there is only one local trivialization map with a cover of $\Omega$ by one set. We will generally denote a metric fiber bundle by $(E, \Omega, \pi, Y)$.

For the remainder of the paper, we fix a metric fiber bundle $(E, \Omega, \pi, Y)$ where $(E, \dist_E)$ and $(\Omega, \dist_\Omega)$  are metric spaces such that $E$ is complete and separable, and~$\Omega$ is complete. Under these conditions, there is a countable, locally finite  subcover~$\{U_j\}_{j\in \N}$ of $\{U_j\}_{j\in \mathcal{J}}$, with local trivializations $\{\Xi_j\}_{j\in \N}$, and we can find a continuous partition of unity $\{\chi_j\}_{j\in \N}$ subordinate to~$\{U_j\}_{j\in \N}$. We also make the technical assumption,
\begin{align}\label{eqn: bounded orbits}
    \text{for each $y\in Y$, the orbit $\{gy\mid g\in G\}$ is a bounded subset of $Y$.}
\end{align}
This holds for trivial bundles, the tangent bundle of any $n$-dimensional Riemannian manifold equipped with the Sasaki metric (see \cite{Sasaki58}*{Section 3}) where $G=O(n)$, and of course any bundle where the diameter of $Y$ is finite or $G$ is compact. 

Then we fix a Borel probability measure $\sigma$ on $\Omega$, and write
\begin{align*}
\mathcal{P}^\sigma(E)
\coloneqq \left\{
\mathfrak{m}\in \mathcal{P}(E)\bigm| \pi_\sharp \mathfrak{m}=\sigma
\right\};
\end{align*}
 $\pi_\sharp\mathfrak{m}\in \mathcal{P}(\Omega)$ is the \emph{push-forward measure} defined by $\pi_\sharp\mathfrak{m}(A)=\mathfrak{m}(\pi^{-1}(A))$ for any Borel $A\subset \Omega$.

We will now  need a result on disintegration of measures (see~\cite{DellacherieMeyer78}*{Chapter III-70 and 72}).
\begin{mythm}[Disintegration Theorem]
\label{thm: disintegration}
Suppose $X$, $\Omega$ are complete, separable metric spaces, $\pi: X\to \Omega$ is a Borel map, and $\mathfrak{m}\in \mathcal{P}(X)$. 
 Then there exists a map $\mathfrak{m}^{\bullet}:\Omega \to \mathcal{P}(X)$, uniquely defined $\pi_\sharp\mathfrak{m}$-a.e., such that 
 for any $A\subset X$ Borel, the real valued function
 \[
 \omega \mapsto \mathfrak{m}^\omega(A)
 \]
 is Borel on $\Omega$, and
\begin{align*}
 \mathfrak{m}(A)=\int_\Omega\mathfrak{m}^\omega(A)d\pi_\sharp\mathfrak{m}(\omega).
\end{align*}
Also for $\pi_\sharp\mathfrak{m}$-a.e. $\omega$, 
\begin{align*}
    \mathfrak{m}^\omega(X\setminus \pi^{-1}(\{\omega\}))=0.
\end{align*}
\end{mythm}
This is the \emph{disintegration of $\mathfrak{m}$ with respect to $\pi$} and we will write  $\mathfrak{m}=\mathfrak{m}^{\bullet}\otimes(\pi_\sharp\mathfrak{m})$. Then for $1\leq p<\infty$, define on our metric fiber bundle $(E, \Omega, \pi, Y)$,
\begin{align*}
    \mathcal{P}^\sigma_p(E)\coloneqq \{\mathfrak{m}=\mathfrak{m}^\bullet\otimes \sigma\in \mathcal{P}^\sigma(E)\mid \mathfrak{m}^\omega\in \mathcal{P}_p(\pi^{-1}(\{\omega\}))\text{ for $\sigma$-a.e. $\omega$}\},
\end{align*}
where the disintegration is with respect to the projection map $\pi$.

We also fix some $y_0\in Y$ and for any Borel $A\subset E$, define
\begin{align}
\label{eqn: delta construction}
     (\delta^\bullet_{E, y_0}\otimes \sigma)(A)\coloneqq \sum_{j\in \N}\int_{\Omega}\chi_j(\omega) (\Xi_{j, \omega})_\sharp\delta^Y_{y_0}(A)d\sigma(\omega).
 \end{align}
 If $\delta^\omega_{E, y_0}\in \mathcal{P}(E)$ is defined for $\sigma$-a.e. $\omega\in \Omega$ as 
\begin{align*}
\delta^\omega_{E,y_0}\coloneqq \sum_{j\in \N}\chi_j(\omega) (\Xi_{j, \omega})_\sharp\delta^Y_{y_0},
\end{align*}
by \cite{KitagawaTakatsu25a}*{Lemma 2.5}, the functional 
in~\eqref{eqn: delta construction} 
belongs to  $\mathcal{P}^\sigma_p(E)$ and its disintegration with respect to $\pi$ is given by $\delta^\bullet_{E, y_0}\otimes \sigma$. Also define 
\begin{align*}
\dist_{E,y_0}^p(\omega, u)
\coloneqq
\sum_{j\in \N}\chi_j(\omega) \dist_E(\Xi_{j, \omega}(y_0),u)^p
\quad \text{for }(\omega, u)\in \Omega\times E.
\end{align*}

Then we can recall the disintegrated Monge--Kantorovich metrics.
\begin{definition}[\cite{KitagawaTakatsu25a}*{Definition 1.2}]
\label{def: scrdist}
Let $1\leq p <\infty$ and $1\leq q\leq\infty$.
Given $\mathfrak{m}$, $\mathfrak{n}\in\mathcal{P}^\sigma_p(E)$,
 define the \emph{disintegrated $(p, q)$-Monge--Kantorovich metric} by
\begin{align*}
 \scrmk_{p,q}(\mathfrak{m}, \mathfrak{n})\coloneqq \left\| \mk_p^{E}(\mathfrak{m}^\bullet, \mathfrak{n}^\bullet)\right\|_{L^q(\sigma)}.
\end{align*}
We also define
\begin{align*}
\mathcal{P}^\sigma_{p,q}(E)\coloneqq \left\{ \mathfrak{m}\in \mathcal{P}^\sigma_p(E) \Biggm|
\scrmk_{p, q}(\delta^\bullet_{E, y_0}\otimes \sigma,\mathfrak{m})<\infty
\right\}.
\end{align*}
\end{definition}
By \cite{KitagawaTakatsu25a}*{Lemmas 2.5 and 2.6} the set $\mathcal{P}^\sigma_{p,q}(E)$ is independent of the choices of $y_0\in Y$, $\{U_j\}_{j\in \N}$, associated local trivializations $\{\Xi_j\}_{j\in \N}$, and partition of unity $\{\chi_j\}_{j\in \N}$. In \cite{KitagawaTakatsu25a}*{Theorem 1.5} it is proved that $(\mathcal{P}^\sigma_{p,q}(E), \scrmk_{p, q})$ is a complete metric space, separable when $q<\infty$, and under some additional hypotheses, is a geodesic space. It is also shown in some cases there is a dual representation for $\scrmk_{p, q}$. 

For a general topological space $X$, we write $C_b(X)$ for the space of bounded continuous functions on $X$ with the supremum norm. We also define the function spaces
\begin{align*}
 \mathcal{X}_p&\coloneqq\left\{\xi\in C(E)\Bigm| \frac{\xi}{1+\dist^p_{E,y_0}(\pi, \cdot)}\in C_0(E)\right\}
\text{ with }\left\| \xi\right\|_{\mathcal{X}_p}\coloneqq\sup_{u\in E}\frac{\left| \xi(u)\right|}{1+\dist^p_{E,y_0}(\pi(u), u)},\\
\mathcal{Z}_{r', \sigma}
&\coloneqq 
\left\{ \zeta\in C_b(\Omega) \bigm| \left\| \zeta\right\|_{L^{r'}(\sigma)}\leq 1,\ \zeta>0
\right\} \text{ for } r'\in [1, \infty];
\end{align*}
 the space $\mathcal{X}_p$ does not depend on the choices of $\{U_j\}_{j\in \N}$, $\{\Xi_j\}_{j\in \N}$, $\{\chi_j\}_{j\in \N}$, and $y_0\in Y$, as shown in \cite{KitagawaTakatsu25a}*{Lemma 2.15}. Finally, for $\lambda \in (0,1]$ and $\xi\in \mathcal{X}_p$, 
we define by $S_{\lambda, p}\xi: E\to (-\infty, \infty]$ by
\begin{align*}
S_{\lambda, p}\xi(u)
\coloneqq \sup_{v\in \pi^{-1}(\{\pi(u)\})} \left(-\lambda \dist_E(u, v)^p-\xi(v) \right)
\quad
\text{for }u\in E.
\end{align*}
Since it is a supremum of continuous functions, $S_{\lambda, p}\xi$ is Borel on $E$ for any $\xi\in \mathcal{X}_p$.
We will omit the subscript $p$ 
and  write $S_\lambda$ for $S_{\lambda, p}$ when there is no possibility of confusion.

With these definitions in hand we can state our main result on barycenters with respect to disintegrated Monge--Kantorovich metrics, which is the following.
\begin{theorem}\label{thm: barycenters}
Fix any $K\in \N$ with $K\geq 2$, $(\lambda_k)_{k=1}^K\in \Lambda_K$, $1\leq p<\infty$,  and $p\leq q\leq \infty$.
   Also let $(E, \Omega, \pi, Y)$ be a metric fiber bundle satisfying the condition~\eqref{eqn: bounded orbits}, with $(E, \dist_E)$ complete and separable, and $(\Omega, \dist_\Omega)$ complete, and let $\sigma\in\mathcal{P}(\Omega)$.
Furthermore, suppose that  $(E, \dist_E)$ is locally compact.
Also fix a subset $\{\mathfrak{m}_k\}_{k=1}^K\subset\mathcal{P}^\sigma_{p,q}(E)$.
        \begin{enumerate}
        \setlength{\leftskip}{-15pt}
        \item\label{thm: disint bary exist} 
If $(Y, \dist_Y)$ has the Heine--Borel property,
then for $\barypower> 0$,
there exists a minimizer of the function 
\begin{align}
\label{eqn: barycenter function}
\mathfrak{n}\mapsto \sum_{k=1}^K \lambda_k \scrmk_{p,q}\left(\mathfrak{m}_k,\mathfrak{n}\right)^\barypower \quad\text{in } \mathcal{P}^\sigma_{p, q}(E).
\end{align}
        \item\label{thm: disint bary duality} 
It holds that 
\begin{align*}
&\inf_{\mathfrak{n}\in \mathcal{P}^\sigma_{p,q}(E)} \sum_{k=1}^K \lambda_k \scrmk_{p,q}\left(\mathfrak{m}_k,\mathfrak{n}\right)^p\\
&=\sup\left\{
 -\sum_{k=1}^K\int_\Omega\zeta_k(\omega)
 \int_E S_{\lambda_k} \xi_k d \mathfrak{m}_k^\omega d\sigma(\omega)\biggm|(\zeta_k, \xi_k)\in \mathcal{Z}_{r', \sigma}\times \mathcal{X}_p \text{ such that }\displaystyle\sum_{k=1}^K\zeta_k\xi_k\equiv 0
\right\},
\end{align*}
where $r'$ is the H\"older conjugate of $r\coloneqq q/p$.
\item\label{thm: disint bary unique} 
Suppose $p>1$, $q<\infty$, $\kappa=p$, and let $Y$ be a complete, connected Riemannian manifold, possibly with boundary. 
Also suppose for some index $1\leq k\leq K$, for $\sigma$-a.e. $\omega$ there is $j\in \N$ with $\omega\in U_j$ such that the measure~$(\Xi_{j, \omega})_\sharp\mathfrak{m}_k^\omega$ is absolutely continuous with respect to the Riemannian volume measure on $Y$. Then minimizers in $\mathcal{P}^\sigma_{p, q}(E)$ of the function~\eqref{eqn: barycenter function} 
are unique, if they exist.
    \end{enumerate}
\end{theorem}
\begin{remark}\label{rem: p=q duality}
When $p=q$ (equivalently, $r'=\infty$), the maximum in Theorem~\ref{thm: barycenters}~\eqref{thm: disint bary duality} is attained when $\zeta_k\equiv 1$ for all $k$. This can be seen as the proof utilizes (through Proposition~\ref{prop: conjugate}) the duality result~\cite{KitagawaTakatsu25a}*{Theorem 1.5 (3)}, where the analogous comment holds.
\end{remark}
\subsection{Discussion of techniques and  
literature}\label{subsec: techniques}
The corresponding results for classical $\mk_2^{\R^n}$-barycenters is due to Agueh and Carlier~\cite{AguehCarlier11}. The instability of disintegration of measures under weak convergence means we are unable to prove existence of $\scrmk_{p, q}$-barycenters by direct compactness methods, thus we have taken the route of using duality in the disintegrated metric setting to prove existence of barycenters. The uniqueness result relies on extracting an appropriate limit of a maximizing sequence in the dual problem, which is by far the most involved proof of the paper. The proof relies on a novel assortment of techniques, which we hope can be of use in other variational problems. Finally, Theorem~\ref{thm: classical barycenters} comes from a quick application of the corresponding results in Theorem~\ref{thm: barycenters} where $\Omega$ is a one point space. We note that the requirement that $Y$ be a Riemannian manifold in Theorem~\ref{thm: barycenters}~\eqref{thm: disint bary unique} and Theorem~\ref{thm: classical barycenters} is only really necessary to obtain Lemmas~\ref{lem: boundary covering} and~\ref{lem: avg c_p-convex}, the remainder of the proof is possible if $Y$ is a space where there is a distinguished class of measures for which all $p$-optimal couplings with left marginals from this class are supported on the graph of an a.e. single valued mapping that can be uniquely determined from a dual potential.  Some existing results on barycenters in similar settings include the results in~\cites{KimPass17, KimPass18, Jiang17, Ohta12}. We note  existing results in the non-manifold setting involve other geometric restrictions (such as Aleksandrov curvature bounds), whereas our result, although restricted to the smooth setting, do not.

We will first prove Theorem~\ref{thm: barycenters}, whose proof is broken down into subsections; Theorem~\ref{thm: classical barycenters} will then follow as a corollary. 
\section{Disintegrated barycenters}\label{sec: barycenters}
In this section, we prove our various claims regarding $\scrmk_{p, q}$-barycenters. For later use, we recall here (see~\cite{KitagawaTakatsu25a}*{(2.4)}): there exists a constant $\tilde{C}>0$ such that
\begin{align}\label{eqn: disint moment estimate}
    \int_{E} \dist_{E,y_0}^p(\omega,v)d\mathfrak{m}^\omega(v)\leq 2^{p-1}(\widetilde{C}+\mk_p^E(\delta^\omega_{E,y_0}, \mathfrak{m}^\omega)^p)
\end{align}
for any $\mathfrak{m}\in \mathcal{P}^\sigma_{p, q}(E)$.
\subsection{Existence of disintegrated barycenters}
We start by proving
Theorem~\ref{thm: barycenters}~\eqref{thm: disint bary exist},
that is,
the existence of $\scrmk_{p, q}$-barycenters. Compared to the case of~$\mk_{p, q}$-barycenters coming from sliced metrics treated in~\cite{KitagawaTakatsu24a}*{Main Theorem (5)}, we lack the continuity needed to apply the direct method, hence we must appeal to the dual problem for~$\scrmk_{p, q}$ to show existence. We will require the fiber $(Y, \dist_Y)$ to be locally compact to apply the duality result~\cite{KitagawaTakatsu25a}*{Theorem 1.5 (3)}, but will actually need the stronger Heine--Borel property on  $(Y, \dist_Y)$. Note that the Heine--Borel property is strictly stronger than local compactness on a complete, separable metric space: the metric space $(\R, \min\{1, d_{\mathbb{R}}\})$ has the same topology as the usual Euclidean one on~$\R$ and is complete and locally compact, but the ball of radius $2$ is all of $\R$ and hence not compact. 
\begin{proof}[Proof of Theorem~\ref{thm: barycenters}~\eqref{thm: disint bary exist}]
Since each $\mathfrak{m}_k\in \mathcal{P}_{p,q}^\sigma(E)$ and 
\[
\sum_{k=1}^K \lambda_k \scrmk_{p,q}\left(\mathfrak{m}_k,\cdot\right)^\barypower
\geq 0 
\quad\text{ on }\mathcal{P}_{p,q}^\sigma(E),
\]
it has a finite infimum and we may take a minimizing sequence $(\mathfrak{n}_\ell)_{\ell\in \mathbb{N}}$, that is 
 \[
 \lim_{\ell\to\infty}\sum_{k=1}^K \lambda_k \scrmk_{p,q}\left(\mathfrak{m}_k,\mathfrak{n}_\ell\right)^\barypower=
 \inf_{\mathfrak{n}\in \mathcal{P}^\sigma_{p, q}(E)}\sum_{k=1}^K \lambda_k \scrmk_{p,q}\left(\mathfrak{m}_k,\mathfrak{n}\right)^\barypower \ \text{ and }\
\sup_{\ell \in \mathbb{N}} \sum_{k=1}^K \lambda_k \scrmk_{p,q}\left(\mathfrak{m}_k,\mathfrak{n}_\ell\right)^\barypower<\infty.
 \]
Since we have
\begin{align*}
\lambda_{1}\scrmk_{p, q}( \mathfrak{n}_\ell, \mathfrak{m}_1)^\barypower
\leq  \sum_{k=1}^K \lambda_k \scrmk_{p,q}\left(\mathfrak{m}_k,\mathfrak{n}_\ell\right)^\barypower
\end{align*}
and $\lambda_1>0$,
we have
\begin{align}
\begin{split}
    \sup_{\ell\in \N}\scrmk_{p, q}(\delta^\bullet_{E,y_0}\otimes \sigma, \mathfrak{n}_\ell)^\barypower
 &\leq  2^{\kappa}
 \left(
\scrmk_{p, q}(\delta^\bullet_{E,y_0}\otimes \sigma, \mathfrak{m}_1)^\barypower
+
 \sup_{\ell\in \N}\scrmk_{p, q}( \mathfrak{n}_\ell, \mathfrak{m}_1)^\barypower
\right)\\
&\leq 2^\barypower\left(
\lambda_1^{-1}\sup_{\ell\in \mathbb{N}} \sum_{k=1}^K \lambda_k \scrmk_{p,q}\left(\mathfrak{m}_k,\mathfrak{n}_\ell\right)^\barypower
+
\scrmk_{p, q}( \delta^\bullet_{E,y_0}\otimes \sigma, \mathfrak{m}_1)^\barypower
\right).\label{eqn: barycenter uniform bound}
\end{split}
\end{align}

We now show that $(\mathfrak{n}_\ell)_{\ell\in \N}$ is tight. Fix an $\varepsilon>0$, since $\sigma$ is a Borel measure, there exists a compact set $K_\Omega\subset \Omega$ such that $\sigma(\Omega\setminus K_\Omega)<\varepsilon/2$. If $q<\infty$, using  Jensen's inequality in the second line below, by~\eqref{eqn: disint moment estimate} and~\eqref{eqn: barycenter uniform bound} we obtain
\begin{align*} \left(\int_E\dist_{E,y_0}^p(\pi(v),v)d\mathfrak{n}_\ell(v)\right)^{\frac{\barypower}{p}}
 &= \left[\int_\Omega\left(\int_E \dist_{E,y_0}^p(\pi(v),v) d\mathfrak{n}^\bullet_\ell(v)\right)d\sigma\right]^{\frac{q}{p}\cdot\frac{\barypower}{q}}\\
 &\leq \left\|\left(\int_E \dist_{E,y_0}^p(\pi(v),v) d\mathfrak{n}^\bullet_\ell(v)\right)^{\frac{1}{p}}\right\|_{L^q(\sigma)}^\barypower\\
 &= \left\|
 2^{\frac{p-1}{p}}\left(\widetilde{C}^{\frac{1}{p}}+\mk_p^E(\delta^\bullet_{E,y_0}, \mathfrak{n}_\ell^\bullet)\right)\right\|_{L^q(\sigma)}^\barypower\\
& \leq 2^{\frac{p-1}{p}\barypower}
 \left(\widetilde{C}^{\frac{\barypower}{p}}+\scrmk_p(\delta^\bullet_{E,y_0}\otimes\sigma, \mathfrak{n}_\ell)^\barypower\right)
\end{align*}
which has a finite upper bound, uniform in $\ell\in \N$ by~\eqref{eqn: barycenter uniform bound}. 
If $q=\infty$, then we can use the trivial inequality
\begin{align*}
\left(\int_\Omega\int_E\dist_{E,y_0}^p(\pi(v),v)d\mathfrak{n}^\bullet_\ell(v)d\sigma\right)^{\frac{\barypower}{p}}
\leq \left\|\int_E \dist_{E,y_0}^p(\pi(v),v)d\mathfrak{n}^\bullet_\ell(v)\right\|_{L^\infty(\sigma)}^{\frac{\barypower}{p}}
\end{align*}
in place of Jensen's inequality to obtain a uniform upper bound. 
Thus in all cases, by Chebyshev's inequality, for  $R>0$ large enough we have
\begin{align*}
    \mathfrak{n}_\ell\left(\left\{v\in E\bigm| \dist_{E,y_0}^p(\pi(v),v)>R\right\}\right)
    &<\frac{\varepsilon}{2},
\end{align*}
hence we find that defining 
\begin{align*}
    K_E\coloneqq \left\{v\in \pi^{-1}(K_\Omega)\bigm| \dist_{E,y_0}^p(\pi(v),v)\leq R\right\},
\end{align*}
we have
\begin{align*}
    \sup_{\ell\in \N}\mathfrak{n}_\ell(E\setminus K_E)
    &<\frac{\varepsilon}{2}+\sup_{\ell\in \N}\mathfrak{n}_\ell(\pi^{-1}(\Omega\setminus K_\Omega))
    =\frac{\varepsilon}{2}+\sigma(\Omega\setminus K_\Omega)<\varepsilon.
\end{align*}
We now show that $K_E$ is a compact subset of $E$. 
Let $(v_\ell)_{\ell\in \N}$ be any sequence in $K_E$ and write $\omega_\ell\coloneqq \pi(v_\ell)$. By compactness of $K_\Omega$, we may pass to a convergent subsequence  $(\omega_\ell)_{\ell\in \N}$
(not relabeled) with limit  $\omega_\infty\in K_\Omega$. 
By local finiteness of $\{U_j\}_{j\in \N}$ and passing to another subsequence, we may assume all $\omega_\ell$ belong to an open neighborhood of $\omega_\infty$ that only meets a finite number of the sets $\{U_{j_i}\}_{i=1}^I$. Passing to another subsequence (and possibly increasing $I$), we may also assume that all $\omega_\ell$ belong to a common set $U_{j_{i_0}}$ for some $1\leq i_0\leq I$ and $\chi_{j_{i_0}}(\omega_\ell)\geq I^{-1}$. Then we have for any $\ell\in\N$
\begin{align*}
    \frac{1}{I}\dist_Y(y_0, \Xi^{-1}_{j_{i_0}, \omega_\ell}(v_\ell))^p
    &=\frac{1}{I}\dist_E(\Xi_{j_{i_0}}(\omega_\ell, y_0), v_\ell)^p\\
    &\leq\sum_{j\in \N}\chi_j(\omega_\ell)\dist_E(\Xi_j(\omega_\ell, y_0), v_\ell)^p
    =\dist_{E,y_0}^p(\pi(v_\ell),v_\ell)   \leq R,
\end{align*}
thus $(\Xi^{-1}_{j_{i_0}, \omega_\ell}(v_\ell))_{\ell\in \N}$ is a bounded sequence in $Y$. Since $Y$ satisfies the Heine--Borel property, we may pass to one final subsequence to assume $\Xi^{-1}_{j_{i_0}, \omega_\ell}(v_\ell)$ converges to some point in $Y$. Thus by continuity of $\Xi_{j_{i_0}}$ we see $v_\ell$ converges to some point in $E$, which again by continuity lies in $K_E$. Hence we see $K_E$ is compact.

Now by Prokhorov's theorem we may pass to a subsequence and assume $(\mathfrak{n}_\ell)_{\ell\in \mathbb{N}}$ converges weakly to some $\mathfrak{n}$ in~$\mathcal{P}^\sigma(E)$. 
Since $Y$ is locally compact, we may apply~\cite{KitagawaTakatsu25a}*{Theorem 1.5 (3)} to obtain for any $\zeta\in \mathcal{Z}_{r', \sigma}$ and $(\Phi, \Psi)\in C_b(E)^2$ such that $-\Phi(u)-\Psi(v) \leq \dist_E(u, v)^p$ whenever $\pi(u)=\pi(v)$,
\begin{align*}
    -\int_E(\zeta\circ\pi)\Phi d(\delta^\bullet_{E,y_0}\otimes \sigma)-\int_E(\zeta\circ\pi)\Psi d\mathfrak{n}
    &=\lim_{\ell\to\infty}\left(-\int_E(\zeta\circ\pi)\Phi d(\delta^\bullet_{E,y_0}\otimes \sigma)-\int_E(\zeta\circ\pi)\Psi d\mathfrak{n}_\ell\right)\\
    &\leq \liminf_{\ell\to\infty}\scrmk_{p, q}(\delta^\bullet_{E,y_0}\otimes \sigma, \mathfrak{n}_\ell)^p,
\end{align*}
where the last term is uniformly bounded in $\zeta$ by~\eqref{eqn: barycenter uniform bound}.
Thus taking a supremum over $\zeta\in \mathcal{Z}_{r', \sigma}$ and all $(\Phi, \Psi)\in C_b(E)^2$ satisfying $-\Phi(u)-\Psi(v) \leq \dist_E(u, v)^p$ for all $\pi(u)=\pi(v)$, by using~\cite{KitagawaTakatsu25a}*{Theorem 1.5 (3)} again we see $\scrmk_{p, q}( \delta^\bullet_{E,y_0}\otimes \sigma, \mathfrak{n})<\infty$, hence $\mathfrak{n}\in \mathcal{P}^\sigma_{p, q}(E)$.

Finally, we can apply~\cite{KitagawaTakatsu25a}*{Corollary 2.24} to obtain 
\begin{align*}
    \sum_{k=1}^K \lambda_k \scrmk_{p,q}\left(\mathfrak{m}_k,\mathfrak{n}\right)^\barypower\leq\liminf_{\ell\to\infty}\sum_{k=1}^K \lambda_k \scrmk_{p,q}\left(\mathfrak{m}_k,\mathfrak{n}_\ell\right)^\barypower.
\end{align*}
\end{proof}

\subsection{Duality for \texorpdfstring{$\scrmk_{p,q}$-}{disintegrated} barycenters}
We now work toward duality for disintegrated barycenters, in the spirit of \cite{AguehCarlier11}*{Proposition~2.2} in the classical Monge--Kantorovich case on $\R^n$ with $p=2$.
In this section, we always assume $1\leq p<\infty$, $1\leq q\leq \infty$. 
Recalling that $r=p/q$ and $r'$ is its H\"older conjugate.

It is well-known (see \cite{Folland99}*{Theorem~7.17}) that if $X$ is a locally compact Hausdorff space, elements of the dual of $C_0(X)$  equipped with the supremum norm can be identified with integration against elements of $\mathcal{M}(X)$, the space of (signed) Radon measures on $X$, moreover the total variation norm is equal to the operator norm. Then we can see 
\begin{align*}
 \mathcal{X}_p^*&=\{\mathfrak{m}\in \mathcal{M}(E)\mid (1+\dist^p_{E, y_0}(\pi, \cdot))\mathfrak{m}\in \mathcal{M}(E)\},
\end{align*}
which is a normed space.
\begin{definition}
Let $\mathfrak{m} \in \mathcal{P}^\sigma(E)$ and $\lambda \in (0,1]$.
For $\eta \in\mathcal{X}_p$ we
define 
\begin{align*}
H_{\lambda, \mathfrak{m}}(\eta)\coloneqq 
\inf\left\{
 \int_\Omega\zeta
 \left(\int_E S_{\lambda} \xi d \mathfrak{m}^\bullet\right) d\sigma
\Bigm| (\zeta, \xi)\in \mathcal{Z}_{r', \sigma}\times \mathcal{X}_p,\ \eta=(\zeta\circ \pi)\xi \right\}.
\end{align*}
\end{definition}
Although $H_{\lambda, \mathfrak{m}}$ also depends on $p$ and $q$, since these are fixed we omit them from the notation.
\begin{lemma}\label{lem: lsc convex}
For $\lambda \in (0,1]$, for any $\mathfrak{m} \in \mathcal{P}^\sigma_{p,1}(E)$ the function $H_{\lambda, \mathfrak{m}}$ is proper and convex on $\mathcal{X}_p$.
\end{lemma}
\begin{proof}
We first prove that
$H_{\lambda, \mathfrak{m}}$ is proper. Since $H_{\lambda, \mathfrak{m}}(0)\leq 0$,
we see $H_{\lambda, \mathfrak{m}}$ is not identically $\infty$. 
Also, for any $\xi \in \mathcal{X}_p$ and $\zeta\in \mathcal{Z}_{r', \sigma}$ we have
$\eta=(\zeta\circ \pi)\xi\in \mathcal{X}_p$ and using~\eqref{eqn: disint moment estimate} in the third line below,
\begin{align*}
  \int_\Omega\zeta\int_ES_{\lambda}\xi(u)d\mathfrak{m}^\bullet(u)d\sigma
 &\geq \int_\Omega\zeta\int_E(-\xi(u))d\mathfrak{m}^\bullet(u)d\sigma
 =-\int_E\eta d\mathfrak{m}\\
 &\geq -\left\|\eta\right\|_{\mathcal{X}_p}\int_E\left(1+\dist^p_{E, y_0}(\pi(u), u)\right) d\mathfrak{m}(u)\\
 &=-2^{p-1}\left\|\eta\right\|_{\mathcal{X}_p}(\widetilde{C}+\scrmk_{p, 1}(\delta^\bullet_{E,y_0}\otimes \sigma, \mathfrak{m}))>-\infty,
\end{align*}
hence $H_{\lambda, \mathfrak{m}}$ is proper.

Next we show $H_{\lambda, \mathfrak{m}}$ is convex. 
Fix $\eta_0, \eta_1 \in\mathcal{X}_p$, 
and for $i=0,1$, let $(\zeta_i, \xi_i)\in \mathcal{Z}_{r', \sigma}\times \mathcal{X}_p$ satisfy $\eta_i=(\zeta_i\circ \pi)\xi_i$.
For $\tau\in (0,1)$, let 
\[
\zeta\coloneqq (1-\tau) \zeta_0+\tau \zeta_1,\qquad
\xi\coloneqq \frac{1}{(\zeta\circ \pi)} \cdot [(1-\tau)(\zeta_0\circ \pi)\xi_0+ \tau(\zeta_1\circ \pi)\xi_1],
\]
then $(1-\tau)\eta_0 +\tau \eta_1=(\zeta\circ \pi)\xi$. 
Moreover, it is clear that $\zeta\in \mathcal{Z}_{r', \sigma}$ and $\xi\in C(E)$. Since
\begin{align*}
\left| \xi\right|
&=\left| \frac{(1-\tau)(\zeta_0\circ \pi)\xi_0+\tau(\zeta_1\circ \pi)\xi_1}{(1-\tau)(\zeta_0\circ \pi)+\tau(\zeta_1\circ \pi)}\right|
\leq 
\max\{\left| \xi_0\right|, \left| \xi_1\right|\}    
\end{align*}
we have $\xi\in \mathcal{X}_p$ as well. 
This yields 
\begin{align}\label{eqn: H convexity}
\begin{split}
H_{\lambda, \mathfrak{m}}((1-\tau)\eta_0+\tau\eta_1)
&\leq 
 \int_\Omega\zeta
 \left(\int_E S_{\lambda} \xi d \mathfrak{m}^\bullet\right) d\sigma\\
&= 
 \int_\Omega
 \int_E
 \sup_{v\in \pi^{-1}(\{\pi(u)\})} \left(-\lambda \dist_E(u, v)^p(\zeta\circ \pi)-\xi(v)(\zeta\circ \pi)\right)d \mathfrak{m}^\bullet(u) d\sigma\\
&=
 \int_\Omega
 \int_E 
 \sup_{v\in \pi^{-1}(\{\pi(u)\})} \left\{-\lambda \dist_E(u, v)^p\left[(1-\tau)(\zeta_0\circ \pi)+\tau(\zeta_1\circ \pi)\right]\right. \\
 &\quad- \left.\left[ (1-\tau)(\zeta_0\circ \pi)\xi_0(v)+ \tau(\zeta_1\circ \pi)\xi_1(v)
 \right]\right\}d \mathfrak{m}^\bullet(u) d\sigma\\ 
&\leq(1-\tau)
 \int_\Omega\zeta_0
 \int_E
 \sup_{v\in \pi^{-1}(\{\pi(u)\})} \left(-\lambda \dist_E(u, v)^p-\xi_0(v) \right)d \mathfrak{m}^\bullet(u) d\sigma\\ 
 &\quad +\tau
  \int_\Omega\zeta_1
  \int_E 
  \sup_{v\in \pi^{-1}(\{\pi(u)\})} \left(-\lambda \dist_E(u, v)^p-\xi_1(v) \right)d \mathfrak{m}^\bullet(u) d\sigma\\ 
&=
(1-\tau) \int_\Omega\zeta_0
 \left(\int_E S_{\lambda} \xi_0 d \mathfrak{m}^\bullet\right) d\sigma+\tau
 \int_\Omega\zeta_1
 \left(\int_E S_{\lambda} \xi_1 d \mathfrak{m}^\bullet\right) d\sigma.
\end{split}
\end{align}
Taking an infimum over admissible $\zeta_i$, $\xi_i$ proves the convexity of $H_{\lambda, \mathfrak{m}}$.
\end{proof}
For $\mathfrak{n}\in \mathcal{X}_p^\ast$,
recall the \emph{Legendre--Fenchel transform} of $H_{\lambda, \mathfrak{m}}$ is defined by
\[
H_{\lambda, \mathfrak{m}}^\ast(\mathfrak{n})
\coloneqq 
\sup_{\eta \in \mathcal{X}_p}
\left(\int_E\eta d\mathfrak{n}-H_{\lambda, \mathfrak{m}}(\eta)\right).
\]
\begin{proposition}\label{prop: conjugate}
Let $\mathfrak{m} \in \mathcal{P}^\sigma_{p,q}(E)$ and $\lambda \in (0,1]$.
If $(E, \dist_E)$ is locally compact, for $\mathfrak{n}\in\mathcal{X}_p^\ast$,
\[
H_{\lambda, \mathfrak{m}}^{\ast}(-\mathfrak{n})
\coloneqq \begin{cases}
\lambda \scrmk_{p,q}(\mathfrak{m}, \mathfrak{n})^p, &\text{if }\mathfrak{n}\in \mathcal{P}^\sigma_{p, q}(E),\\
\infty, &\text{else}.
\end{cases}
\]
\end{proposition}
\begin{proof}
First suppose $\mathfrak{n}\in\mathcal{P}^\sigma(E)$,
then by~\cite{KitagawaTakatsu25a}*{Theorem 1.5 (3)},
\begin{align*}
H_{\lambda, \mathfrak{m}}^{\ast}(-\mathfrak{n})
&=\sup_{\eta\in \mathcal{X}_p}\left(-\int_E\eta d\mathfrak{n}-H_{\lambda, \mathfrak{m}}(\eta)\right) \\
&=
 \sup_{\eta\in \mathcal{X}_p}\sup_{
 \substack{
 (\zeta, \xi)\in \mathcal{Z}_{r', \sigma}\times \mathcal{X}_p,\\ \eta=(\zeta\circ \pi)\xi}}\int_\Omega
\left(
 -\int_E
\eta(v)
d\mathfrak{n}^\bullet(v)-
\zeta\int_E
S_{\lambda}\xi(u) d\mathfrak{m}^\bullet(u) \right)d\sigma \\
&=\sup_{
 \substack{
 (\zeta, \xi)\in \mathcal{Z}_{r', \sigma}\times \mathcal{X}_p,\\ \eta=(\zeta\circ \pi)\xi}}\left[-\int_\Omega
\zeta\cdot\left(
 \int_E
\xi(v)
d\mathfrak{n}^\bullet(v)+
\int_E
S_{\lambda}\xi(u) d\mathfrak{m}^\bullet(u) \right) d\sigma\right]
= 
 \lambda\scrmk_{p, q}(\mathfrak{m}, \mathfrak{n})^p,
\end{align*}
note that since $\mathfrak{m}\in\mathcal{P}^\sigma_{p, q}(E)$, we have $\scrmk_{p, q}(\mathfrak{m}, \mathfrak{n})=\infty$ if $\mathfrak{n}\not\in\mathcal{P}^\sigma_{p, q}(E)$.

We now handle the case of $\mathfrak{n}\notin \mathcal{P}^\sigma(E)$. First suppose 
$\mathfrak{n}\in \mathcal{X}^\ast_p$ and $\pi_\sharp\mathfrak{n}\neq\sigma$. In this case, there exists some $\phi\in C_b(\Omega)$ such that 
\[
\int_\Omega \phi d\sigma\neq \int_E\phi(\pi(v)) d\mathfrak{n}(v).
\]
For $C\in \R$, define $\eta_{C\phi}\in\mathcal{X}_p$ by $\eta_{C\phi}(u)\coloneqq -C\phi(\pi(u))$. 
Then we have 
\begin{align*}
S_{\lambda}\eta_{C, \phi}(u)
&=\sup_{v\in \pi^{-1}(\{\pi(u)\})}(-\lambda \dist_E(u, v)^p+C\phi(\pi(v)))\\
&=\sup_{v\in \pi^{-1}(\{\pi(u)\})}(-\lambda \dist_E(u, v)^p+C\phi(\pi(u)))=C\phi(\pi(u)).
\end{align*}
Since we can decompose $\eta_{C, \phi}=(\zeta\circ \pi)\xi$ where $\zeta\equiv 1$ and $\xi=\eta_{C, \phi}$, we calculate
\begin{align*}
 H^*_{\lambda, \mathfrak{m}}(-\mathfrak{n})&\geq \sup_{C\in \R}\left(-\int_E
\eta_{C\phi}
d\mathfrak{n}-
\int_\Omega\int_E
S_{\lambda}\eta_{C\phi} d\mathfrak{m}^\omega d\sigma(\omega)\right)\\
&= \sup_{C\in\mathbb{R}} C
\left(\int_E\phi(\pi(v)) d\mathfrak{n}(v)-\int_E \phi(\pi(u)) d\mathfrak{m}(u)\right)\\
&= \sup_{C\in\mathbb{R}} C
\left(\int_E\phi(\pi(v)) d\mathfrak{n}(v)-\int_\Omega \phi d\sigma\right)
=\infty.
\end{align*}
Now suppose $\mathfrak{n}\in \mathcal{X}^\ast_p$ is not nonnegative. 
Here, $\mathfrak{n}$ is said to be nonnegative if 
$\mathfrak{n}(E')\geq0$ for any measurable set $E'\subset E$,
hence there exists some 
$\eta\in \mathcal{X}_p$ such that $\eta\geq 0$ everywhere and
\[
-\int_E \eta d \mathfrak{n}
>0.
\]
Then it is clear from the definition that $-S_{\lambda}(C\eta)\geq 0$ on $E$ for any constant $C>0$, hence we can again calculate
\begin{align*}
H_{\lambda, \mathfrak{m}}^*(-\mathfrak{n})
\geq \sup_{C>0}
\left( -\int_E C\eta d \mathfrak{n}- \int_\Omega\int_E
S_{\lambda}(C\eta) d\mathfrak{m}^\omega d\sigma(\omega) \right)
\geq \sup_{C>0}\left(-C\int_E \eta d \mathfrak{n} \right)=\infty.
\end{align*}
\end{proof}
We are now ready to prove our duality result for $\scrmk_{p, q}$-barycenters.
\begin{proof}[Proof of Theorem~\ref{thm: barycenters}~\eqref{thm: disint bary duality}]
Let $\mathfrak{n}\in \mathcal{P}^\sigma_{p,q}(E)$
and $(\eta_k)_{k=1}^K$ a collection in $\mathcal{X}_p$ such that 
\[
\sum_{k=1}^K \eta_k\equiv 0.
\]
For each $k$ fix
$(\zeta_k, \xi_k)\in \mathcal{Z}_{r', \sigma}\times \mathcal{X}_p$ such that $\eta_k=(\zeta_k\circ\pi)\xi_k$ (which is always possible, for example by taking $\zeta_k\equiv 1$, $\xi_k\equiv \eta_k$).
Since $\mk_p^E(\delta^\omega_{E,y_0}, \mathfrak{n}^\omega)<\infty$ for $\sigma$-a.e.~$\omega$, for all $k$, using~\eqref{eqn: disint moment estimate} we have
\begin{align*}
    \left|\int_E\xi_k(u)d\mathfrak{n}^\omega(u)\right|
    &\leq\left\| \xi_k\right\|_{\mathcal{X}_p}\int_E\left(1+\dist^p_{E, y_0}(\pi(u),u)^p\right)d\mathfrak{n}^\omega(u)\\
    \leq \left\| \xi_k\right\|_{\mathcal{X}_p}(1+2^{p-1}(\widetilde{C}+\mk_p^E(\delta^\omega_{E,y_0},\mathfrak{n}^\omega)^p))<\infty.
\end{align*}
Then for such $\omega\in \Omega$  and $1\leq k \leq K$, 
we can first integrate the inequality
\[
\lambda_k\dist_E(u, v)^p
\geq
-S_{\lambda_k} \xi_k(u)-\xi_k(v)
\]
which holds for any $u$, $v\in E$ such that $\pi(u)=\pi(v)$, against a $p$-optimal coupling between $\mathfrak{m}_k^\omega$ and~$\mathfrak{n}^\omega$, 
then multiply by~$\zeta_k(\omega)$ and integrate in $\omega$ against $\sigma$ to obtain
\begin{align*}
\lambda_k \scrmk_{p,q}(\mathfrak{m}_k, \mathfrak{n})^p
\geq\lambda_k\int_\Omega\zeta_k \mk_p^E(\mathfrak{m}_k^\bullet, \mathfrak{n}^\bullet)^pd\sigma  
&\geq 
-\int_\Omega \zeta_k \int_E S_{\lambda_k} \xi_k  d\mathfrak{m}_k^\bullet d\sigma  
-
\int_\Omega \zeta_k \int_E \xi_kd\mathfrak{n}^\bullet d\sigma \\
&=-\int_\Omega \zeta_k \int_E S_{\lambda_k} \xi_k  d\mathfrak{m}_k^\bullet d\sigma -\int_E \eta_k d\mathfrak{n}.
\end{align*}
Since $\sum_{k=1}^K \eta_k\equiv 0$, 
taking a supremum over all such pairs $(\zeta_k, \xi_k)$, then summing over $1\leq k \leq K$ in the above inequality gives
\[
\sum_{k=1}^K\lambda_k \scrmk_{p,q}(\mathfrak{m}_k,\mathfrak{n})^p
\geq 
-\sum_{k=1}^K H_{\lambda_k, \mathfrak{m}_k}(\eta_k)-\int_E \sum_{k=1}^K\eta_k d\mathfrak{n}=-\sum_{k=1}^K H_{\lambda_k, \mathfrak{m}_k}(\eta_k).
\]
Thus it follows that
\begin{align}\label{inequality1}
\begin{split}
\inf_{\mathfrak{n}\in \mathcal{P}^\sigma_{p,q}(E)} \sum_{k=1}^K \lambda_k \scrmk_{p,q}\left(\mathfrak{m}_k,\mathfrak{n}\right)^p
\geq 
\sup\left\{ 
-\sum_{k=1}^K H_{\lambda_k, \mathfrak{m}_k}(\eta_k)
\Bigm| 
\sum_{k=1}^K \eta_k\equiv 0,\ \eta_k \in \mathcal{X}_p\right\}.
\end{split}
\end{align}

Let us now show the reverse inequality.
It follows from Proposition~\ref{prop: conjugate} that
\begin{align*}
\inf_{\mathfrak{n} \in \mathcal{P}^\sigma_{p,q}(E)} \sum_{k=1}^K \lambda_k \scrmk_{p,q}\left(\mathfrak{m}_k,\mathfrak{n}\right)^p
=
\inf_{\mathfrak{n} \in \mathcal{P}^\sigma_{p,q}(E)}
\sum_{k=1}^K\lambda_k \scrmk_{p,q}(\mathfrak{m}_k, \mathfrak{n})^{p}
=
\inf_{\mathfrak{n} \in \mathcal{P}^\sigma_{p,q}(E)}
\sum_{k=1}^K H_{\lambda_k, \mathfrak{m}_k}^\ast(-\mathfrak{n}).
\end{align*}
Define the function $H$ on $\mathcal{X}_p$
as the infimal convolution of $\{H_{\lambda_k, \mathfrak{m}_k}\}_{k=1}^K$, that is, defined for 
$\eta\in\mathcal{X}_p$ by 
\[
H(\eta)
\coloneqq 
\inf
\left\{
\sum_{k=1}^K H_{\lambda_k, \mathfrak{m}_k}(\eta_k)
\Bigm| 
\sum_{k=1}^K \eta_k\equiv \eta,\ \eta_k \in \mathcal{X}_p
\right\}.
\]
Then \eqref{inequality1} implies
\begin{align}\label{inequality2}
\inf_{\mathfrak{n}\in \mathcal{P}^\sigma_{p,q}(E)} \sum_{k=1}^K \lambda_k \scrmk_{p,q}\left(\mathfrak{m}_k,\mathfrak{n}\right)^p
\geq -H(0).
\end{align}
Note that $H$ is convex 
since each of $\{H_{\lambda_k, \mathfrak{m}_k}\}_{k=1}^K$ is proper and convex by  Lemma~\ref{lem: lsc convex},  and then by~\cite{BorweinVanderwerff19}*{Lemma 4.4.15}
the Legendre--Fenchel transform of~$H$ satisfies 
\begin{align}\label{eqn: inf convolution conjugate}
H^{\ast}(\mathfrak{n})
=\sum_{k=1}^K H_{\lambda_k, \mathfrak{m}_k}^\ast(\mathfrak{n})
\quad \text{for }\mathfrak{n}\in \mathcal{X}_{p}^\ast.
\end{align}
Let $\mathcal{X}^{\ast\ast}_p$ be the dual of $\mathcal{X}^\ast_p$
and regard $\mathcal{X}_p$ as a subset of $\mathcal{X}^{\ast\ast}_p$ under the natural isometric embedding.
For $\mathfrak{f} \in \mathcal{X}^{\ast\ast}_p$, 
the Legendre--Fenchel transform of $H^\ast$ on $\mathcal{X}^{\ast\ast}_p$ 
is given by
\[
H^{\ast\ast}(\mathfrak{f})
\coloneqq 
\sup_{\mathfrak{n} \in \mathcal{X}_p^\ast}
\left( \mathfrak{f}(\mathfrak{n} )
-H^\ast(\mathfrak{n})
\right).
\]
Then we observe from Proposition~\ref{prop: conjugate} and \eqref{eqn: inf convolution conjugate} that 
\begin{align}\label{inequality4}
\begin{split}
-H^{\ast\ast}(0)
=\inf_{\mathfrak{n}\in\mathcal{X}^\ast_p}H^\ast(-\mathfrak{n})
=\inf_{\mathfrak{n}\in\mathcal{X}^\ast_p}
\sum_{k=1}^K H_{\lambda_k, \mathfrak{m}_k}^\ast(-\mathfrak{n})
=
\inf_{\mathfrak{n}\in\mathcal{P}^\sigma_{p,q}(E)}
\sum_{k=1}^K \lambda_k \scrmk_{p,q}( \mathfrak{m}_k,\mathfrak{n})^p.
\end{split}
\end{align}
Thus by \eqref{inequality2} and \eqref{inequality4} it is enough to show $H^{\ast\ast}(0)=H(0)$.

To this end, first note by Proposition~\ref{prop: conjugate} combined with \eqref{eqn: inf convolution conjugate} we see  
\[
H^*(-\delta^\bullet_{E,y_0}\otimes \sigma)=\sum_{k=1}^K \lambda_k \scrmk_{p, q}(\delta^\bullet_{E,y_0}\otimes \sigma, \mathfrak{m}_k)<\infty. 
\]
Thus since its Legendre--Fenchel transform is not identically $\infty$,
we see $H$ never takes the value $-\infty$. At the same time using  $H_{\lambda, \mathfrak{m}}(0)\leq 0$,
\begin{align*}
 H(0)&\leq \sum_{k=1}^K H_{\lambda_k, \mathfrak{m}_k}(0)\leq 0,
\end{align*}
hence $H$ is not identically $\infty$, in particular it is proper. 

Recall each $\lambda_k$ is strictly positive by assumption.  
Suppose $\eta\in \mathcal{X}_p$ with
\[
\left\| \eta\right\|_{\mathcal{X}_p}\leq 2^{1-p}\cdot K\cdot \min_{1 \leq k\leq K}\lambda_k.
\]
Then, using that 
\begin{align*}
2^{1-p}\dist_{E,y_0}^p(\omega,v) -  \dist_E(u, v)^p\leq
\dist_{E,y_0}^p(\omega,u),
\end{align*}
followed by~\eqref{eqn: disint moment estimate} in the calculation below,
\begin{align*}
    H(\eta)
    &\leq\sum_{k=1}^KH_{\lambda_k, \mathfrak{m}_{k}}(K^{-1}\eta)\\
    &\leq\sum_{k=1}^K
     \int_\Omega\int_E S_{\lambda_k}\left(K^{-1}\eta\right) d\mathfrak{m}_k^\omega d\sigma(\omega)\\
    &\leq\sum_{k=1}^K\int_\Omega\int_E\sup_{v\in \pi^{-1}(\{\pi(u)\})}\left(-\lambda_k\dist_E(u, v)^p
    +K^{-1}\left\| \eta\right\|_{\mathcal{X}_p}(1+\dist_{E,y_0}^p(\pi(v),v))\right) d\mathfrak{m}_k^\omega(u) d\sigma(\omega)\\
&
\leq\sum_{k=1}^K \lambda_k\int_\Omega\int_E\sup_{v\in \pi^{-1}(\{\pi(u)\})}\left(-\dist_E(u, v)^p
    +2^{1-p}(1+\dist_{E,y_0}^p(\pi(v),v))\right) d\mathfrak{m}_k^\omega(u) d\sigma(\omega)\\
    &\leq\sum_{k=1}^K\lambda_k\int_\Omega\int_E \sup_{v\in \pi^{-1}(\{\pi(u)\})}\left(2^{1-p}+\dist_{E,y_0}^p(\pi(v),u)\right) d\mathfrak{m}_k^\omega(u) d\sigma(\omega)\\
    &\leq\sum_{k=1}^K\lambda_k \left[2^{1-p}+2^{p-1}\left(\widetilde{C}+\scrmk_{p, q}(\delta^\bullet_{E,y_0}\otimes\sigma, \mathfrak{m}_k)^p\right)\right]<\infty,
\end{align*}
proving that $H$ is bounded from above in a neighborhood of $0$. Thus by \cite{BorweinVanderwerff19}*{Proposition 4.1.4 and Proposition 4.4.2 (a)}, we obtain $H^{\ast\ast}(0)=H(0)$, finishing the proof.
\end{proof}

\subsection{Uniqueness of disintegrated barycenters}
In this final subsection, we prove that when the fiber $Y$ is a Riemannian manifold, $\scrmk_{p, q}$-barycenters are unique under some absolute continuity conditions, when $p>1$ and $q<\infty$. This subsection will be the most involved.

We start by noting that in the case $q=\infty$, it is possible to construct many examples where $\scrmk_{p, \infty}$-barycenters are not unique; the next examples includes all cases when $\sigma$ is not a delta measure and the fiber $Y$ is a connected, complete Riemannian manifolds of any kind (with or without boundary).

At this point, we define the collection~$\{V_j\}_{j\in \N}$ consisting of mutually disjoint Borel sets by 
\begin{align}\label{eqn: disjoint cover}
 V_1\coloneqq \{\omega\in \Omega\mid \chi_1(\omega)>0\},\quad V_j\coloneqq \{\omega\in \Omega\mid \chi_j(\omega)>0\}\setminus\bigcup_{j'=1}^{j-1}V_{j'},\ j\geq 2,
\end{align} 
by construction $\chi_j>0$ on $V_j$ and $V_j\subset U_j$ for each $j\in \N$. Since $\{\chi_j\}_{j\in \N}$ is a partition of unity, we see that $\{V_j\}_{j\in \N}$ covers $\Omega$. This disjoint cover will also play a role in the proof of uniqueness.
\begin{example}\label{ex: q=infinity not unique}
Let $1< p<\infty$ (the case $p=1$ may have nonuniqueness for other reasons, see Example~\ref{ex: p=1 nonunique} below), 
 make the same assumptions as in Theorem~\ref{thm: barycenters}~\eqref{thm: disint bary exist}, and also assume $(Y, \dist_Y)$ is any geodesic space. 
Also take two distinct elements $\mu_0,\mu_1\in\mathcal{P}_p(Y)$, 
and assume there exists a measurable $\Omega'\subset \Omega$ with $0<\sigma(\Omega')<1$, 
and define for any Borel $A\subset E$,
\begin{align*}
        (\mathfrak{m}^\bullet_k\otimes\sigma)(A)&\coloneqq
        \begin{dcases}
            \sum_{j\in \N}\int_{V_j}(\Xi_{j, \bullet})_\sharp\mu_0(A)d\sigma,&\text{if }1\leq k \leq K-1,\\
            \sum_{j\in \N}\left(\int_{V_j\cap \Omega'}(\Xi_{j, \bullet})_\sharp\mu_0(A)d\sigma +\int_{V_j\setminus \Omega'}(\Xi_{j, \bullet})_\sharp\mu_1(A)d\sigma\right),&\text{if }k=K,
        \end{dcases}
    \end{align*}
 where we recall that $\{V_j\}_{j\in \N}$ is defined by~\eqref{eqn: disjoint cover}. By~\cite{KitagawaTakatsu25a}*{Lemma 2.5}, each of these are elements of $\mathcal{P}^\sigma_{p, \infty}(E)$, with disintegrations with respect to $\pi$ given by $\mathfrak{m}_k=\mathfrak{m}_k^\bullet\otimes\sigma$ where
    \begin{align*}
        \mathfrak{m}_k^\omega&\coloneqq
        \begin{dcases}
            \sum_{j\in \N}\mathds{1}_{V_j}(\omega)(\Xi_{j, \omega})_\sharp\mu_0,&\text{if }1\leq k \leq K-1,\\
            \sum_{j\in \N}\mathds{1}_{V_j}(\omega)(\Xi_{j, \omega})_\sharp\left(\mathds{1}_{\Omega'}(\omega)\mu_0 +\mathds{1}_{\Omega\setminus \Omega'}(\omega)\mu_1\right),&\text{if }k=K.
        \end{dcases}
    \end{align*}
   
For any $\mathfrak{n}\in \mathcal{P}_{p,\infty}^{\sigma}(E)$,  $\barypower>0$, and $(\lambda_k)_{k=1}^K\in \Lambda_K$,
we calculate
\begin{align*}
\sum_{k=1}^K \lambda_k \scrmk_{p,\infty}\left(\mathfrak{m}_k,\mathfrak{n}\right)^\barypower
&=
(1-\lambda_K)
\scrmk_{p,\infty}\left(\mathfrak{m}_1,\mathfrak{n}\right)^\barypower
+
\lambda_K\scrmk_{p,\infty}\left(\mathfrak{m}_K,\mathfrak{n}\right)^\barypower\\
&\geq 
(1-\lambda_K)
\esssup_{\omega\notin \Omega'}\mk_{p}^E(\mathfrak{m}^\omega_1, \mathfrak{n}^\omega)^\barypower+\lambda_K
\esssup_{\omega\notin \Omega'}\mk_{p}^E(\mathfrak{m}^\omega_K, \mathfrak{n}^\omega)^\barypower.
\end{align*}

Let $\nu\in \mathcal{P}_p(Y)$ be a minimizer of $(1-\lambda_K)\mk_p^Y(\mu_0, \cdot)^\barypower+ \lambda_K\mk_p^Y(\mu_1, \cdot)^\barypower$, 
then for each $\omega\notin\Omega'$, if $j_0$ is the unique index such that $\omega\in V_{j_0}$,
\begin{align*}
&(1-\lambda_K)
\mk_{p}^E(\mathfrak{m}^\omega_1, \mathfrak{n}^\omega)^\barypower+\lambda_K
\mk_{p}^E(\mathfrak{m}^\omega_K, \mathfrak{n}^\omega)^\barypower\\
&=(1-\lambda_K)
\mk_{p}^Y( \mu_0, (\Xi_{j_0,\omega}^{-1})_\sharp \mathfrak{n}^\omega)^\barypower
+
\lambda_K
\mk_{p}^Y( \mu_1, (\Xi_{j_0,\omega}^{-1})_\sharp \mathfrak{n}^\omega)^\barypower\\
&\geq 
(1-\lambda_K)\mk_{p}^Y( \mu_0, \nu)^\barypower
+
\lambda_K \mk_{p}^Y( \mu_1, \nu)^\barypower
\end{align*}
hence if $\mu\in \mathcal{P}_p(Y)$ satisfies 
\begin{align*}
\mk_p^Y(\mu_0,\mu)^\barypower \leq  (1-\lambda_K)
 \mk_{p}^Y(\mu_0, \nu)^\barypower
 +\lambda_K\mk_{p}^Y(\mu_1, \nu)^\barypower,
\end{align*} 
the measure 
\begin{align*}
    \sum_{j\in \N}\mathds{1}_{V_j}(\Xi_{j, \bullet})_\sharp\left( \mathds{1}_{\Omega'}\mu+\mathds{1}_{\Omega\setminus \Omega'}\nu \right)\otimes \sigma,
\end{align*} (which belongs to $\mathcal{P}^\sigma_{p, q}(E)$ by~\cite{KitagawaTakatsu25a}*{Lemma 2.5}) 
is a minimizer of
\[
\mathfrak{n}\mapsto \sum_{k=1}^K \lambda_k \scrmk_{p,\infty}\left(\mathfrak{m}_k,\mathfrak{n}\right)^\barypower
\quad \text{on } \mathcal{P}^\sigma_{p, q}(E).
\]
Thus since $\lambda_K\neq 0$, $1$, this yields infinitely many possible minimizers. 
\end{example}
Also, we can see $\scrmk_{1, q}$-barycenters may not be unique due to nonuniqueness of $\mk_1^Y$-barycenters. Recall that for any pair of sets $X$ and $Y$, $c: X\times Y\to (-\infty, \infty]$, and $\phi: X\to (-\infty, \infty]$, the \emph{$c$-transform of $\phi$} is the function $\phi^c: Y\to (-\infty, \infty]$ defined by 
\begin{align*}
    \phi^c(y)\coloneqq\sup_{x\in X}(-c(x, y)-\phi(x)).
\end{align*}
\begin{example}\label{ex: p=1 nonunique}
    Let $(\mu_k)_{k=1}^K\in\mathcal{P}_p(Y)^K$ to be determined 
    and define the measures 
    $\mathfrak{m}_k\in\mathcal{P}^\sigma_{p, q}(E)$ for $1\leq k\leq K$ by 
\begin{align*}    
\mathfrak{m}_k\coloneqq \left\{\left(\sum_{j\in \N}\mathds{1}_{V_j}(\Xi_{j, \bullet})\right)_\sharp\mu_k\right\}\otimes\sigma.
\end{align*}
For $(\lambda_k)_{k=1}^K \in \Lambda_K$, by  convexity of the $L^q(\sigma)$ norm, for any $\mathfrak{n}\in \mathcal{P}^\sigma_{p, q}(E)$ we have
    \begin{align*}
        \sum_{k=1}^K\lambda_k\scrmk_{p, q}(\mathfrak{m}_k, \mathfrak{n})
        &\geq \left\| \lambda_k\sum_{k=1}^K\mk_p^E\left(\sum_{j\in \N}\mathds{1}_{V_j}(\Xi_{j, \bullet})_\sharp\mu_k, \mathfrak{n}^\bullet\right)\right\|_{L^q(\sigma)}.
    \end{align*}
    For any measure of the form 
    \[\mathfrak{n}\coloneqq \left(\sum_{j\in \N}\mathds{1}_{V_j}(\Xi_{j, \bullet})_\sharp\nu_0\right)\otimes \sigma    
\]
where $\nu_0\in \mathcal{P}_p(Y)$, if $j_0$ is the unique index such that $\omega\in V_{j_0}$ we have
    \begin{align*}
        \mk_p^E\left(\sum_{j\in \N}\mathds{1}_{V_j}(\Xi_{j, \omega})_\sharp\mu_k, \mathfrak{n}^\omega\right)
        &=\mk_p^E((\Xi_{j_0, \bullet})_\sharp\mu_k, (\Xi_{j_0, \bullet})_\sharp\nu_0))=\mk_p^Y(\mu_k, \nu_0).
    \end{align*}
    Hence if $\nu_0$ is an $\mk_p^Y$-barycenter, we see $\mathfrak{n}$ will be a $\scrmk_{p, q}$-barycenter, thus if $(\mu_k)_{k=1}^K$ can be chosen in a way that there exist nonunique $\mk_p^Y$-barycenters, this will lead to nonuniqueness of $\scrmk_{p, q}$-barycenters as well.

For $p=1$, it is strongly suspected that such configurations yielding nonunique barycenters exist for various $(\lambda_k)_{k=1}^K$, we give such an example in the case of $Y=\R$ with the measures $\mu_k$ absolutely continuous, and $\lambda_k\equiv K^{-1}$ where $K$ is even, which incidentally, relies on our duality result Theorem~\ref{thm: classical barycenters}. Let $\mathcal{H}^1$ be $1$-dimensional Hausdorff measure, and define 
    \begin{align*}
        \nu_0\coloneqq \mathcal{H}^1|_{[-2,-1]},\qquad
        \nu_1\coloneqq \mathcal{H}^1|_{[1,2]},\qquad
        \mu_k\coloneqq 
        \begin{cases} 
\nu_0, & \text{if $k$ even},\\
\nu_1, & \text{if $k$ odd}.
        \end{cases}\\
\end{align*}
Then we calculate
\begin{align*}
    \frac{1}{K}\sum_{k=1}^K  \scrmk_{1,q}\left(\mathfrak{m}_k,\nu_0\otimes \sigma\right)
    &=\frac{1}{K}\sum_{k=1}^K\mk_1^\R(\mu_k, \nu_0)
    \leq \frac{1}{K}\sum_{\text{$k$ odd}}\int_1^2\left| t-(t-3)\right| dt=\frac{3}{2},\\
    \frac{1}{K}\sum_{k=1}^K  \scrmk_{1,q}\left(\mathfrak{m}_k,\nu_1\otimes \sigma\right)
    &=\frac{1}{K}\sum_{k=1}^K\mk_1^\R(\mu_k, \nu_1)
    \leq \frac{1}{K}\sum_{\text{$k$ even}}\int_{-2}^{-1}\left| t-(t+3)\right| dt=\frac{3}{2}.
\end{align*}
Now define $\phi: \R\to \R$ by
\begin{align*}
\phi(t)\coloneqq 
    \begin{cases}
        -4-t,&\text{if }-4\leq t< -2,\\
        t,&\text{if }-2\leq t< 2,\\
        4-t,&\text{if }2\leq t\leq 4,\\
        0,&\text{else}.
    \end{cases}
\end{align*}
Since $\phi$ is $1$-Lipschitz, it is classical that $\phi^{\dist_\R}=-\phi$, then if we define
\begin{align*}
    \phi_k(t)&\coloneqq
    \begin{dcases}
            -\frac{\phi(t)}{K},& \text{if $k$ even},\\
            \frac{\phi(t)}{K},& \text{if $k$ odd},
        \end{dcases}
\end{align*}
we have  
\begin{align*}
\sum_{k=1}^K\phi_k &\equiv 0,\\
-\sum_{k=1}^K\int_{\R}\phi_k^{\lambda_k\dist_\R} d\mu_k
    &=-\sum_{\text{$k$ even}}\int_{-2}^{-1}\frac{\phi(t)}{K}dt+\sum_{\text{$k$ odd}}\int_1^2\frac{\phi(t)}{K}dt
    =\frac{1}{2}\left(\int_1^2tdt-\int_{-2}^{-1}tdt\right)
    =\frac{3}{2}.
\end{align*}
By Theorem~\ref{thm: classical barycenters}~\eqref{thm: classical bary duality} we see that both $\nu_0$ and $\nu_1$ are $\mk_1^\R$-barycenters.
\end{example}
For the remainder of the section $Y$ will be a complete, connected Riemannian manifold, possibly with boundary, 
and $\dist_Y$ (resp.\,$\Vol_Y$) will be the Riemannian distance function (resp.\,volume measure). 
We will also write
\begin{align*}
    \inj_1(y)&\coloneqq\min\left\{1, 
    \sup\left\{r>0\biggm|\! 
 \begin{tabular}{l}$\exp_y$ is a diffeomorphism on the open ball \\of radius $r$ centered at $0\in T_y(Y\setminus\partial Y)$ \end{tabular}
\right\}
\right\}
\quad\text{for }y\in Y\setminus\partial Y,\\
    \inj(A)&\coloneqq\inf_{y\in A}\inj_1(y)
\quad \text{for }A\subset Y\setminus\partial Y.
\end{align*}
We will write $B_r(y)$ and $\overline{B}_r(y)$ for the \emph{open} and \emph{closed} ball, respectively, of radius $r>0$ centered at $y\in Y\setminus \partial Y$. We will only consider balls in the space $Y$ in this paper. 
Although $Y\setminus\partial Y$ may not be complete, by \cite{Boumal23}*{Lemmas 10.90 and 10.91}, we have $\inj(K)>0$ for any compact $K\subset Y\setminus\partial Y$.
 
First we show a very simple lemma on covering boundaries of geodesic balls.
\begin{lemma}\label{lem: boundary covering}
For any compact set $K\subset Y\setminus \partial Y$ and $0<r<\inj(K)/2$, there is an $N\in \N$ depending only on $K$ and $r$ such that for any $y\in K$, there exists a set of points $\{y_i\}_{i=1}^N\subset \overline{B}_{5r/4}(y)\setminus B_{3r/4}(y)$ such that $\{B_{r/2}(y_i)\}_{i=1}^N$ is a cover of $\partial B_r(y)$.
\end{lemma}
\begin{proof}
    Suppose the lemma does not hold, then there exists a sequence $(\tilde{y}_\ell)_{\ell\in \N}\subset K$ such that no collection of $\ell$ or fewer open balls of radius $r/2$ with centers in $\overline{B}_{5r/4}(\tilde{y}_\ell)\setminus B_{3r/4}(\tilde{y}_\ell)$ is a cover of $\partial B_r(\tilde{y}_\ell)$.
    By compactness of $K$, we may pass to a convergent subsequence $(\tilde{y}_\ell)_{\ell\in\mathbb{N}}$ (not relabeled) with limit $\tilde{y}_\infty \in K$. 
    Now, also by compactness, for some $N\in \N$ there is a cover $\{B_{r/2}(y_i)\}_{i=1}^N$ of $\overline{B}_{9r/8}(\tilde{y}_\infty)\setminus B_{7r/8}(\tilde{y}_\infty)$ with $y_i\in \overline{B}_{9r/8}(\tilde{y}_\infty)\setminus B_{7r/8}(\tilde{y}_\infty)$ for $1\leq i\leq N$. 
    Since $r<\inj(K)/2$ and $\tilde{y}_\ell\in K$, we see that $y\in \partial B_r(\tilde{y}_\ell)$ implies $\dist_Y(\tilde{y}_\ell, y)=r$.
Then by the triangle inequality, for $\ell>N$ satisfying $\dist_Y(\tilde{y}_\ell, \tilde{y}_\infty)<r/8$,  we have
\[
\partial B_r(\tilde{y}_\ell)\subset \overline{B}_{9r/8}(\tilde{y}_\infty)\setminus B_{7r/8}(\tilde{y}_\infty)\] 
while each $y_i\in \overline{B}_{5r/4}(\tilde{y}_\ell)\setminus B_{3r/4}(\tilde{y}_\ell)$, a contradiction.
\end{proof}
It is well known that local boundedness for a $\lambda \dist_Y^p$-convex function translates to a Lipschitz bound. To show convergence of a maximizing sequence in the disintegrated barycenter dual problem from Theorem~\ref{thm: barycenters}~\eqref{thm: disint bary duality}, we will need to consider sequences of \emph{averages} constructed from the maximizing sequence. When $p=2$, the average of $\dist_{\R^n}^2$-transforms is also a $\dist_{\R^n}^2$-transform, but this does not hold for $p\neq 2$ or on more general manifolds $Y$. Thus in the next lemma, we will prove that under certain conditions, local Lipschitzness of the average of $\dist_Y^p$-transforms also follows from boundedness. The following lemma is stated in more generality than will be needed later.
\begin{lemma}\label{lem: avg c_p-convex}
    Fix $\lambda\in (0, 1]$, $R>0$, and suppose  $(g_m)_{m\in \N}$ is a sequence such that the functions 
 $f_m\coloneqq g_m^{\lambda \dist_Y^p}$ are bounded uniformly in $m\in\mathbb{N}$ on $\overline{B}_R(y_0)$. 
 If there exists an increasing sequence $(M_\ell)_{\ell\in \N}\subset \N$, and $\lambda_{\ell, m}\geq 0$ for each $\ell\in \N$ and $1\leq m\leq M_\ell$, and $C_1$, $C_2>0$ such that    
\begin{align*}
 \sup_{\ell\in \N}\frac{1}{M_\ell}\sum_{m=1}^{M_\ell}\lambda_{\ell, m}\leq C_1,\qquad
       \sup_{t\in \overline{B}_R(y_0)}  \frac{1}{M_\ell}\sum_{m=1}^{M_\ell}\lambda_{\ell, m}\left| f_m(t)\right|\leq C_2,
    \end{align*} 
    then the sequence
    \begin{align*}
    \left(\frac{1}{M_\ell}\sum_{m=1}^{M_\ell}\lambda_{\ell, m}f_m\right)_{\ell\in\N}
    \end{align*}
     is uniformly Lipschitz on $\{y\in \overline{B}_{R/2}(y_0)\mid \dist_Y(y, \partial Y)\geq 2R^{-1}\}$. 
\end{lemma}
\begin{proof}
We can assume that $\lambda=1$ as 
\[
g_m^{\lambda \dist_Y^p}=
\lambda (\lambda^{-1} g_m)^{\dist_Y^p}.
\]
Since the result follows from \cite{Santambrogio15}*{Proposition 3.1} when $p=1$, assume $1<p<\infty$. 
Let $N\in \N$ be from applying Lemma~\ref{lem: boundary covering} with the choice of 
the set
\[
K\coloneqq \left\{y\in \overline{B}_R(y_0)\mid \dist_Y(y, \partial Y)\geq 2R^{-1}\right\},
\]
compact in $Y\setminus \partial Y$, and  $r\coloneqq \min\{\inj(K), R\}/2$. 
 Now let us write
 \begin{align*}
     \widehat{B}\coloneqq
     \left\{y\in \overline{B}_{R/2}(y_0)\mid \dist_Y(y, \partial Y)\geq 2R^{-1}\right\}.
 \end{align*} 
Fix $t \in \widehat{B}$ and $\varepsilon>0$, 
 then since each $f_m$ is finite on $\overline{B}_R(y_0)$, 
 for each $m$ there exists $s_{m, t}\in Y$ such that
 \[
  f_m(t)\leq - \dist_Y(t, s_{m, t})^p-g_m(s_{m, t})+\varepsilon.
 \] 
Then for any $t'\in Y$,  we have
\begin{align}
\begin{split}
f_m(t')+\varepsilon
&\geq - \dist_Y(t', s_{m, t})^p-g_m(s_{m, t})+\varepsilon\\
&\geq - \dist_Y(t', s_{m, t})^p+ \dist_Y(t, s_{m, t})^p+f_m(t)\\
&\geq p \dist_Y(t', s_{m, t})^{p-1}(\dist_Y(t, s_{m, t})-\dist_Y(t', s_{m, t}))+f_m(t).\label{eqn: f_m est}
\end{split}
\end{align}
Now let $\{B_{r/2}(y_i)\}_{i=1}^N$ be a cover of $\partial B_{r}(t)$ with 
$y_i\in \overline{B}_{5r/4}(t)\setminus B_{3r/4}(t)$, which exists from the application of Lemma~\ref{lem: boundary covering} above. By completeness and connectedness, there is at least one minimal, unit speed geodesic $\gamma_{m, t}$ from $t$ to $s_{m, t}$. For $1\leq i\leq N$, define
\begin{align*}
    B_i\coloneqq
    \begin{dcases}
        B_{r/2}(y_1),&\text{if }i=1,\\
        B_{r/2}(y_i) \setminus \bigcup_{i'=1}^{i-1}B_{r/2}(y_{i'}),&\text{if }2\leq i\leq N,
    \end{dcases}\qquad
    I_i\coloneqq\{m\in \N\mid \gamma_{m, t}(r)\in B_i\text{ and }s_{m, t}\not\in B_{2r}(t)\}.
\end{align*}
Then for $m\in I_i$, using that $r<\inj(K)/2$ and $t\in K$ we have
\begin{align*}
    \dist_Y(t, s_{m, t})-\dist_Y(y_i, s_{m, t})
    &\geq \dist_Y(t, s_{m, t})-\dist_Y(\gamma_{m, t}(r), s_{m, t})-\dist_Y(\gamma_{m, t}(r), y_i)\\
    &\geq \dist_Y(t, \gamma_{m, t}(r))-\frac{r}{2}= \frac{r}{2}.
\end{align*}
Using this we can calculate for each $1\leq i\leq N$, by taking $t'=y_i$ in~\eqref{eqn: f_m est} and noting that each $y_i\in \overline{B}_R(y_0)$,
 \begin{align*}
     C_2+\varepsilon
     &\geq \frac{1}{M_\ell}\sum_{m=1}^{M_\ell}\lambda_{\ell, m}\left| f_m(y_i)\right|+\varepsilon\\
     &\geq \frac{1}{M_\ell}\sum_{\substack{1\leq m \leq M_\ell,\\ m\in I_i}}\lambda_{\ell, m}\left[ p \dist_Y(y_i, s_{m, t})^{p-1}\left(\dist_Y(t, s_{m, t})-\dist_Y(y_i, s_{m, t})\right)+f_m(t)\right]\\
     &\geq \frac{pr}{2M_\ell}\sum_{\substack{1\leq m \leq M_\ell,\\ m\in I_i}}\lambda_{\ell, m}\dist_Y(y_i, s_{m, t})^{p-1}-\frac{1}{M_\ell}\sum_{m=1}^{M_\ell}\lambda_{\ell, m}\left| f_m(t)\right|\\
     &\geq\frac{pr}{2M_\ell}\sum_{\substack{1\leq m \leq M_\ell,\\ m\in I_i}}\lambda_{\ell, m}\left[ 2^{-p+1}\dist_Y(t'', s_{m, t})^{p-1}-\dist_Y(t'', y_i)^{p-1}\right]  -C_2\\
     &\geq\frac{pr}{2M_\ell}\sum_{\substack{1\leq m \leq M_\ell,\\ m\in I_i}}\lambda_{\ell, m}\left[ 2^{-p+1}\dist_Y(t'', s_{m, t})^{p-1}-(2R)^{p-1}\right]  -C_2
 \end{align*}
for any $t''\in \widehat{B}$. Hence, for $t_1$, $t_2\in \widehat{B}$, we find
\begin{align*}
        \frac{1}{M_\ell}\sum_{m=1}^{M_\ell} \lambda_{\ell, m}f_m(t_1)-\frac{1}{M_\ell}\sum_{m=1}^{M_\ell} \lambda_{\ell, m}f_m(t_2)
        &\leq \frac{1}{M_\ell}\sum_{m=1}^{M_\ell} \lambda_{\ell, m}\left( \dist_Y(t_2, s_{m, t_1})^p- \dist_Y(t_1, s_{m, t_1})^p+\varepsilon\right)\\
        &\leq \frac{p}{M_\ell}\sum_{m=1}^{M_\ell}\lambda_{\ell, m}\max\{\dist_Y(t_1, s_{m, t_1})^{p-1}, \dist_Y(t_2, s_{m, t_1})^{p-1}\}\\
        &\qquad\cdot\left| \dist_Y(t_2, s_{m, t_1})-\dist_Y(t_1, s_{m, t_1})\right|+\varepsilon C_1\\
        &\leq \frac{2^{p}}{r}(2C_2+\varepsilon+2^{p-2}prR^{p-1}C_1)\dist_Y(t_1, t_2)+\varepsilon C_1,
    \end{align*}
thus taking $\varepsilon$ to $0$ and then reversing the roles of $t_1$ and $t_2$ yields the uniform Lipschitz bound on~$\widehat{B}$.
\end{proof}
The above lemma also immediately gives an analogue of \cite{GangboMcCann96}*{Corollary C.5} which we will have use for later.
\begin{corollary}\label{cor: bounded is lipschitz}
    Fix $\lambda\in (0, 1]$ and suppose $R>0$. 
For a function $g$ on~$Y$, 
if  $f\coloneqq g^{\lambda \dist_Y^p}$ is bounded on $\overline{B}_r(y_0)$, then it is uniformly Lipschitz on 
the set $\{y\in \overline{B}_{R/2}(y_0)\mid \dist_Y(y, \partial Y)\geq 2R^{-1}\}$.
\end{corollary}
\begin{proof}
    Simply apply Lemma~\ref{lem: avg c_p-convex} with $f_m\equiv f$ and $\lambda_{\ell, m}\equiv 1$.
\end{proof}
Next we show a compactness result which will be applied to a  maximizing sequence in the dual problem for $\scrmk_{p, q}$-barycenters as described in Theorem~\ref{thm: barycenters}~\eqref{thm: disint bary duality}. Recall that $\{V_j\}_{j\in \N}$ is a partition of $\Omega$ into Borel sets defined by~\eqref{eqn: disjoint cover}.
\begin{lemma}\label{lem: subsequence of maximizing sequence}
    For each $1\leq k\leq K$, let $(\xi_{k, m})_{m\in \N}$ and $(\zeta_{k, m})_{m\in \N}$ be sequences in $\mathcal{X}_p$ and $\mathcal{Z}_{r', \sigma}$ respectively, write $\eta_{k, m}(v)\coloneqq \zeta_{k, m}(\pi(v))\xi_{k, m}(v)$, and assume that, 
\begin{align}
    &\sum_{j\in \N}\chi_j\mathds{1}_{V_j}\xi_{k, m}(\Xi_j(\cdot, y_0))=0,
    &&\text{for $1\leq k\leq K-1$ and $m\in \N$},
    \label{eqn: zero at y0}\\
    &\sum_{k=1}^K\eta_{k, m}\equiv 0,
    &&\text{for $m\in \N$},
\label{eqn: eta_k sum to zero}\\
    &\xi_{k, m}=S_{\lambda_k}(S_{\lambda_k}\xi_{k, m}),
    &&\text{for $1\leq k\leq K-1$ and $m\in \N$},
    \label{eqn: double transform}\\
    &\inf_{m\in\N}\left(-\sum_{k=1}^K\int_\Omega\zeta_{k, m}(\omega)\int_E S_{\lambda_k}\xi_{k, m} d\mathfrak{m}^\omega_kd\sigma(\omega)\right)>-\infty,\label{eqn: inf lower bound}\\
    &\zeta_{k, m}\xrightarrow{m\to\infty}\zeta_k\quad \text{in }L^{r'}(\sigma),
    &&\text{for $1\leq k\leq K-1$}.
\label{eqn: zetas weak converge}
\end{align}
Additionally if $p=q$, assume that $\zeta_{k, m}\equiv 1$ for all $1\leq k\leq K$ and $m\in \N$. Then there is a Borel set $\Omega'\subset \Omega$ with $\sigma(\Omega')=1$, and for each $1\leq k\leq K$, subsequences of $(\eta_{k, m})_{m\in \N}$, $(\zeta_{k, m})_{m\in \N}$ (not relabeled), such that there is a Borel function $\eta_k: E\to \R$ whose restriction to $\pi^{-1}(\{\omega\})$ is continuous for all $\omega\in \Omega'$, and writing
\begin{align*}
\eta_{k, M}^{\avg}(v)&\coloneqq\frac{1}{M}\sum_{m=1}^M\eta_{k, m}(v),\qquad
    \zeta^{\avg}_{k, M}(\omega)\coloneqq \frac{1}{M}\sum_{m=1}^{M} \zeta_{k, m}(\omega),
\end{align*}
we have for each $1\leq k\leq K$,
\begin{align}
    &\lim_{M\to\infty} \eta_{k, M}^{\avg}(v)=\eta_k(v),&& \text{for all } v\in \pi^{-1}(\Omega'),\label{eqn: etas converge everywhere}\\
    &\lim_{M\to\infty}\zeta^{\avg}_{k, M}(\omega)= \zeta_k(\omega),&&\text{for all } \omega\in  \Omega',\label{eqn: zeta avg converge}\\
    &\lim_{M\to\infty}\lVert \zeta^{\avg}_{k, M}-\zeta_k\rVert_{L^{r'}(\sigma)}=0,\label{eqn: zeta avg converge in L^r'}
    \end{align}
    and
    \begin{align}\label{eqn: etas sum to zero}
\sum_{k=1}^K  \eta_k&\equiv 0. 
\end{align}
Moreover, the convergence of $\eta_{k, M}^{\avg}$ to $\eta_k$ is uniform on the sets
\[
\overline{B}_{\omega, \ell}
\coloneqq 
\left\{\Xi_j(\omega, y)\mid y\in \overline{B}_\ell(y_0),\ \dist_Y(y, \partial Y)\geq 2\ell^{-1}\right\},
\]
for each $\ell\in \N$ and $\omega\in \Omega'$, where $j\in \N$ is the unique index such that $\omega\in V_j$.
\end{lemma}
\begin{proof}
For any $1\leq k\leq K$, $m\in \N$, and fixed $u\in E$, from~\eqref{eqn: zero at y0} we have
\begin{align}
\begin{split}
    -S_{\lambda_k}\xi_{k, m}(u)
    &=\inf_{v\in\pi^{-1}(\{\pi(u)\})}(\lambda_k \dist_E(u, v)^p+\xi_{k, m}(v))\\
    &\leq \sum_{j\in \N}\chi_j(\pi(u))\mathds{1}_{V_j}(\pi(u))\left(\lambda_k \dist_E(u, \Xi_{j, \pi(u)}(y_0))^p+\xi_{k, m}(\Xi_{j, \pi(u)}(y_0))\right)\\&\leq\lambda_k \dist_{E,y_0}^p(\pi(u),u),\label{eqn: transform y0 bound}
\end{split}
\end{align}
thus for any $m\in \N$ and $1\leq k\leq K$, using~\eqref{eqn: disint moment estimate},
\begin{align*}
    \sum_{i\neq k}\left(-\int_\Omega\zeta_{i, m}\int_E S_{\lambda_i}\xi_{i, m}d\mathfrak{m}^\bullet_i d\sigma\right)
    &\leq 
    \sum_{i\neq k}\lambda_i\left(\int_\Omega\zeta_{i, m}\int_E  \dist^p_{E, y_0}(\cdot, u) d\mathfrak{m}^\bullet_i(u) d\sigma\right)\\
    &\leq 
    \sum_{i\neq k}\lambda_i\left\| \zeta_{i, m}\right\|_{L^{r'}(\sigma)}\left\Vert 2^{p-1}(\widetilde{C}+\mk_p^E(\delta^\bullet_{E,y_0}, \mathfrak{m}^\bullet_i)^p\right\|_{L^r(\sigma)}\\
    &\leq 
    2^{p-1}\sum_{i\neq k}\lambda_i\left(\widetilde{C}+\scrmk_p(\delta^\bullet_{E,y_0}\otimes \sigma, \mathfrak{m}_i)^p\right).
\end{align*}
Hence for any $m\in \N$,
\begin{align}
\begin{split}
    &-\int_\Omega\zeta_{k, m}\int_E S_{\lambda_k}\xi_{k, m}d\mathfrak{m}^\bullet d\sigma\\
    &\geq -\sum_{k'=1}^K \int_\Omega\zeta_{k', m}\int_E S_{\lambda_{k'}}\xi_{k', m}d\mathfrak{m}^\bullet d\sigma
    -\sum_{i\neq k}\left(-\int_\Omega\zeta_{i, m}\int_E S_{\lambda_i}\xi_{i, m}d\mathfrak{m}^\bullet_i d\sigma\right)\\
    &\geq -\sum_{k'=1}^K \int_\Omega\zeta_{k', m}\int_E S_{\lambda_{k'}}\xi_{k', m}d\mathfrak{m}^\bullet d\sigma- 2^{p-1}\sum_{i\neq k}\lambda_i\left(\widetilde{C}+\scrmk_p(\delta^\bullet_{E,y_0}\otimes \sigma, \mathfrak{m}_i)^p\right)\\
    &\eqqcolon C,
\label{eqn: single term lower bound}
\end{split}
\end{align}
here $C$ is not $-\infty$ and independent of $m\in \N$ by~\eqref{eqn: inf lower bound}. 

Now for a fixed~$\omega\in \Omega$ and $u$, $v\in \pi^{-1}(\{\omega\})$ we can integrate the inequality 
\begin{align}
\begin{split}\label{etasd}
\eta_{k, m}(v)
&\geq -\zeta_{k, m}(\omega)S_{\lambda_k}\xi_{k, m}(u)-\lambda_k\zeta_{k, m}(\omega)\dist_E(u, v)^p\\
&\geq  -\zeta_{k, m}(\omega)S_{\lambda_k}\xi_{k, m}(u)-2^{p-1}\zeta_{k, m}(\omega)(\dist_{E,y_0}^p(\omega,u)+\dist_{E,y_0}^p(\omega,v))
\end{split}
\end{align} 
with respect to $ \mathfrak{m}_k^\omega\otimes \mathfrak{n}^\omega(u, v)$ for any $ \mathfrak{n}\in\mathcal{P}^\sigma_{p, q}(E)$, then integrate against $\sigma$ with respect to~$\omega$, and using that each~$\mathfrak{n}^\omega$ is nonnegative and has total mass one along with~\eqref{eqn: disint moment estimate} and~\eqref{eqn: single term lower bound}, we thus obtain 
\begin{align*}
\int_E\eta_{k, m}d\mathfrak{n}
&\geq -\int_\Omega\zeta_{k, m}\int_ES_{\lambda_k}\xi_{k, m} d\mathfrak{m}^\bullet_kd\sigma\\
    &\quad -2^{p-1}\left(\int_\Omega\zeta_{k, m}\left[\int_E \dist^p_{E,y_0}(\cdot,u)d\mathfrak{m}_k^\bullet(u) d\sigma+\int_E \dist^p_{E,y_0}(\cdot,v)d\mathfrak{n}^\bullet(v)\right]d\sigma\right)\\
    &\geq C
    -2^{2p-2}\left\| \zeta_{k, m}\right\|_{L^{r'}(\sigma)}\left\|2\widetilde{C}+\mk_p^E(\delta^\bullet_{E,y_0}, \mathfrak{m}^\bullet_k)^p+\mk_p^E(\delta^\bullet_{E,y_0}, \mathfrak{n}^\bullet)^p\right\|_{L^r(\sigma)}\\
    &\geq C
    -2^{2p-2}\left(2\widetilde{C}+\max_{1\leq i \leq K} \scrmk_{p,q}(\delta^\bullet_{E,y_0}\otimes \sigma, \mathfrak{m}_i)^p+\scrmk_{p,q}(\delta^\bullet_{E,y_0}\otimes \sigma, \mathfrak{n})^p
    \right).
\end{align*}
Combining with~\eqref{eqn: eta_k sum to zero},  
there exists a constant $C'>0$ independent of $m\in \N$, $1\leq k\leq K$, and $\mathfrak{n}\in \mathcal{P}^\sigma_{p, q}(E)$ such that 
\begin{align}
\label{eqn: max seq uniform bound}
\left|\int_E\eta_{k, m}d\mathfrak{n}\right|
\leq C'(1+\scrmk_{p,q}(\delta^\bullet_{E,y_0}\otimes \sigma, \mathfrak{n})^p).
\end{align}
Now define for $\delta>0$, $1\leq k\leq K$, $j$, $m\in \N$, and $\omega\in \Omega$,
\begin{align*}
    I^{\delta, \omega}_{k,  \ell, m}\coloneqq \left \{t\in \overline{B}_\ell(y_0)\Biggm|
    \sum_{j\in \N}\mathds{1}_{V_j}(\omega)\eta_{k, m}(\Xi_{j, \omega}(t))
    \geq \sup_{t'\in \overline{B}_\ell(y_0)}\sum_{j\in \N}\mathds{1}_{V_j}(\omega)\eta_{k, m}(\Xi_{j, \omega}(t'))-\delta
        \right\}.
\end{align*}
Since 
\[
t\mapsto \sum_{j\in \N}\mathds{1}_{V_j}(\omega)\eta_{k, m}(\Xi_{j, \omega}(t))
\]
is continuous on $Y$ for any fixed $\omega$, we must have $\Vol_Y(I^{\delta, \omega}_{k,  \ell, m})>0$, so we can define
\begin{align*}
    \mu^\omega_{\delta, k, \ell, m}&\coloneqq\frac{\mathds{1}_{I^{\delta, \omega}_{k,  \ell, m}}}{\Vol_Y(I^{\delta, \omega}_{k,  \ell, m})}\Vol_Y\in \mathcal{P}(Y).
\end{align*}
By the continuity of each $\Xi_j$ and $\eta_{k, m}$, we can see the set 
\[
\left\{(\omega, t)\in \Omega\times Y\mid t\in I^{\delta, \omega}_{k,  \ell, m}\right\}
\]
is a Borel subset of $\Omega\times Y$. Thus the function 
\begin{align*}
(\omega, t)\mapsto \mathds{1}_{I^{\delta, \omega}_{k,  \ell, m}}(t)
\end{align*}
is Borel on $\Omega\times Y$, and by Tonelli's theorem the function $\omega\mapsto \Vol_Y(I^{\delta, \omega}_{k,  \ell, m})$ is Borel on $\Omega$. 
Now fix any Borel  $A\subset E$, then as a composition of a Borel function $\mathds{1}_A$
with a continuous map $\Xi_j$, the function $(\omega, t)\mapsto\mathds{1}_A(\Xi_{j, \omega}(t))$ is Borel on $U_j\times Y$ (endowed with the subspace metric), then the function 
\begin{align*}
(\omega, t)\mapsto\sum_{j\in\N}\mathds{1}_{V_j}(\omega)\cdot\mathds{1}_{I^{\delta, \omega}_{k,  \ell, m}}(t)\cdot\mathds{1}_A(\Xi_{j, \omega}(t))
\end{align*}
is Borel on $\Omega\times Y$. 
Thus combining the above, if we define
\begin{align*}
 \mathfrak{n}^\omega_{\delta, k, \ell, m}\coloneqq \sum_{j\in \N}\mathds{1}_{V_j}(\omega)(\Xi_{j, \omega})_\sharp  \mu^\omega_{\delta, k, \ell, m}
\end{align*}
again by Tonelli's theorem we see $\omega \mapsto \mathfrak{n}^\omega_{\delta, k, \ell, m}(A)$ is Borel on $\Omega$ for any Borel $A\subset E$, hence  $\mathfrak{n}_{\delta, k, j, \ell, m}\coloneqq \mathfrak{n}_{\delta, k, j, \ell, m}^\bullet \otimes \sigma$ is well-defined and belongs to~$\mathcal{P}^\sigma(E)$ by~\cite{KuratowskiRyll-Nardzewski65}*{Corollary 1}. Also if $\omega\in V_{j_0}$ for some $j_0\in \N$,
\begin{align}
\begin{split}
&\mk^E_p(\delta^\omega_{E,y_0}, \mathfrak{n}^\omega_{\delta, k, \ell, m})^p\\
&\leq\sum_{j'\in \N}\chi_{j'}(\omega)\mk_p^E((\Xi_{j', \omega})_\sharp\delta^Y_{y_0},(\Xi_{j_0, \omega})_\sharp \mu^\omega_{\delta, k, \ell, m})^p\\
&\leq 2^{p-1}\sum_{j'\in \N}\chi_{j'}(\omega)\Big(\mk_p^E((\Xi_{j', \omega})_\sharp\delta^Y_{y_0}, (\Xi_{j_0, \omega})_\sharp\delta^Y_{y_0})^p+\mk_p^E((\Xi_{j_0, \omega})_\sharp\delta^Y_{y_0}, (\Xi_{j_0, \omega})_\sharp \mu^\omega_{\delta, k, \ell, m})^p\Big)\\
&= 2^{p-1}\sum_{j'\in \N}\chi_{j'}(\omega)\Bigg(\dist_Y(y_0, g^{j'}_{j_0}(\omega)y_0)^p
+\Vol_Y(I^{\delta, \omega}_{k,  \ell, m})^{-1}\int_{I^{\delta, \omega}_{m, k, \ell}}d_{y_0}(t)^pd\Vol_Y(t)\Bigg)\\
&\leq 2^{p-1}\left(\sum_{j'\in \N}\chi_{j'}(\omega)\dist_Y(y_0, g^{j'}_{j_0}(\omega)y_0)^p+\ell^p\right),\label{eqn: n delta moment bound}
\end{split}
\end{align}
which is bounded independent of $\omega$ and $j_0$ by~\eqref{eqn: bounded orbits}, hence $\mathfrak{n}_{\delta, k, \ell, m}\in \mathcal{P}_{p, q}^\sigma(E)$.
 Then we find
\begin{align*}
&\int_\Omega \sup_{t\in\overline{B}_\ell(y_0)}\left(\sum_{j\in \N}\mathds{1}_{V_j}(\omega)\eta_{k, m}(\Xi_{j, \omega}(t))\right)d\sigma(\omega)-\delta\\
&\leq \int_\Omega\frac{1}{\Vol_Y(I^{\delta, \omega}_{k,  \ell, m})}\int_{I^{\delta, \omega}_{k,  \ell, m}}\sum_{j\in \N}\mathds{1}_{V_j}(\omega)\eta_{k, m}(\Xi_{j, \omega}(t))d\Vol_Y(t)  d\sigma(\omega)
=\int_\Omega\int_E \eta_{k, m}d\mathfrak{n}^\omega_{\delta, k, \ell, m}d\sigma(\omega)\\
&\leq C_\ell
\end{align*}
for some $C_\ell>0$ independent of $k$, $m$, and $\delta$ by~\eqref{eqn: max seq uniform bound} and~\eqref{eqn: n delta moment bound}.
We may replace $\max$ with $\min$ and change the direction of the inequality in the definition of $I^{\delta, \omega}_{k, m, \ell}$, then replace $\sup$ with $\inf$ above 
 to obtain the analogous inequality 
\begin{align*}
\int_\Omega\inf_{t\in\overline{B}_\ell(y_0)}\left(\sum_{j\in \N}\mathds{1}_{V_j}(\omega)\eta_{k, m}(\Xi_{j, \omega}(t))\right)d\sigma(\omega)+\delta
\geq
-C_\ell.
\end{align*}
For a fixed $\omega\in \Omega$, if $j_0\in \N$ is the unique index for which $\omega\in V_{j_0}$, using~\eqref{eqn: zero at y0}
\begin{align*}
    0=\sum_{j\in \N}\chi_j(\omega)\mathds{1}_{V_j}(\omega)\xi_{k, m}(\Xi_j(\omega, y_0))=\chi_{j_0}(\omega)\xi_{k, m}(\Xi_{j_0}(\omega, y_0)),
\end{align*}
since $\chi_{j_0}(\omega)>0$ by construction of $\{V_j\}_{j\in \N}$, we must have
\begin{align}\label{eqn: partition version zero at y0}
    \sum_{j\in \N}\mathds{1}_{V_j}(\omega)\xi_{k, m}(\Xi_j(\omega, y_0))
    &=\xi_{k, m}(\Xi_{j_0}(\omega, y_0))=0.
\end{align}
In particular,
\begin{align*}
    \inf_{t\in\overline{B}_\ell(y_0)}\left(\sum_{j\in \N}\mathds{1}_{V_j}(\omega)\eta_{k, m}(\Xi_{j, \omega}(t))\right)
    \leq 0\leq \sup_{t\in\overline{B}_\ell(y_0)}\left(\sum_{j\in \N}\mathds{1}_{V_j}(\omega)\eta_{k, m}(\Xi_{j, \omega}(t))\right)
\end{align*}
for any $\omega\in \Omega$, and $m$, $\ell\in \N$. 
Thus taking $\delta$ to $0$ in the two resulting inequalities above and using H\"older's inequality yields
\begin{align}
\begin{split}
   & \int_\Omega\left\| \sum_{j\in \N}\mathds{1}_{V_j}(\omega)(\eta_{k, m}\circ\Xi_{j, \omega})\right\|_{L^2(\overline{B}_\ell(y_0))}d\sigma(\omega)\\
    &\leq  \Vol_Y(\overline{B}_\ell(y_0))^{\frac{1}{2}}\int_\Omega\sup_{t\in \overline{B}_\ell(y_0)}\left| \sum_{j\in \N}\mathds{1}_{V_j}(\omega)\eta_{k, m}(\Xi_{j, \omega}(t))\right| d\sigma(\omega)\\
&\leq C_\ell\Vol_Y(\overline{B}_\ell(y_0))^{\frac{1}{2}},\label{eqn: L infinity bound}
\end{split}
\end{align}
where the reference measure on $L^2(\overline{B}_\ell(y_0))$ is $\Vol_Y$. 
This implies that for each $\ell\in \N$ and $1\leq k\leq K$, the sequence 
\begin{align}\label{eqn: bounded seq}
    \left(\omega\mapsto \sum_{j\in \N}\mathds{1}_{V_j}(\omega)(\eta_{k, m}\circ\Xi_{j, \omega})\right)_{m\in \N}
\end{align}
is bounded in the Bochner--Lebesgue space $L^1(\sigma; L^2(\overline{B}_\ell(y_0)))$.
As the space $L^2(\overline{B}_\ell(y_0))$ is a Hilbert space, we may repeatedly apply~\cite{Guessous97}*{Theorem 3.1} 
along with a diagonalization argument to obtain a subsequence of~\eqref{eqn: bounded seq} (which we do not relabel) 
with the property that: there exists a function 
\[
\tilde{\eta}_k: \Omega\times Y\to \R
\quad\text{with}\quad \omega\mapsto\tilde{\eta}_k(\omega, \cdot)\vert_{\overline{B}_\ell(y_0)}\in L^1(\sigma; L^2(\overline{B}_\ell(y_0)))
\]
for each $\ell\in \N$, and for any further (not relabeled) subsequence there is a $\sigma$-null Borel set $\mathcal{N}_1\subset \Omega$ such that for all $\ell\in \N$ and $\omega\in \Omega\setminus \mathcal{N}_1$,
\begin{align}\label{eqn: L^2 conv}
    \lim_{M\to \infty}\left\| \tilde{\eta}^{\avg}_{k, M}(\omega, \cdot)-\tilde{\eta}_k(\omega, \cdot)\right\|_{L^2(\overline{B}_\ell(y_0))}= 0,
\quad\text{where }\ 
\tilde{\eta}^{\avg}_{k, M}(\omega, t)\coloneqq \frac{1}{M}\sum_{m=1}^M\sum_{j\in \N}\mathds{1}_{V_j}(\omega)\eta_{k, m}(\Xi_{j, \omega}(t)).
\end{align}
By~\eqref{eqn: L infinity bound} and since 
\[
 \sup_{m\in \N}\left\| \zeta_{k, m}\right\|_{L^1(\sigma)}\leq \sup_{m\in \N}\left\| \zeta_{k, m}\right\|_{L^{r'}(\sigma)}\leq 1
\]
we can also apply the real valued Koml{\'o}s' theorem (see \cite{Komlos67}*{Theorem 1a}) for each $1\leq k\leq K$ and $\ell\in \N$ to the sequences 
\begin{align*}
\left(\omega\mapsto\sup_{t'\in\overline{B}_\ell(y_0)}\left| \sum_{j\in \N}\mathds{1}_{V_j}(\omega)\eta_{k, m}(\Xi_{j, \omega}(t'))\right|\right)_{m\in \N}    
\end{align*}
and $(\zeta_{k,m})_{m\in \N}$,
and make yet another diagonalization argument to assume there exists a $\sigma$-null Borel set $\mathcal{N}_2$ such that for all $\ell\in \N$, $1\leq k \leq K$, and $\omega\in \mathcal{N}_2$,
\begin{align*}
&   \lim_{M\to\infty}\frac{1}{M}\sum_{m=1}^M\sup_{t'\in\overline{B}_\ell(y_0)}\left| \sum_{j\in \N}\mathds{1}_{V_j}(\omega)\eta_{k, m}(\Xi_{j, \omega}(t'))\right| \text{ converges},
\end{align*}
and \eqref{eqn: zeta avg converge} holds. If $p<q$, by the Banach--Saks theorem we may pass to another subsequence of $(\zeta_{k,m})_{m\in \N}$ to assume that $\zeta_{k, M}^{\avg}$ also converges in $L^{r'}(\sigma)$, necessarily to $\zeta_k$ by~\eqref{eqn: zetas weak converge}, while if $p=q$ we already have $\zeta_{k, M}^{\avg}\equiv 1$ for all $M$, proving~\eqref{eqn: zeta avg converge in L^r'}. 
 
 Now fix an arbitrary increasing sequence $(M_{\ell'})_{\ell'\in \mathbb{N}}\subset \N$ and 
 \[
 \omega\in \Omega'\coloneqq \Omega\setminus(\mathcal{N}_1\cup \mathcal{N}_2),
 \]
 where $\Omega'$ is Borel.
 By \eqref{eqn: L^2 conv} we may pass to yet another subsequence to assume for some $\Vol_Y$-null set $\mathcal{N}(\omega)\subset Y$,
\begin{align*}
    \lim_{\ell'\to\infty}\tilde{\eta}_{k, M_{\ell'}}^{\avg}(\omega, t)= \tilde{\eta}_k(\omega, t),
    \quad\text{for all }t\in Y\setminus\mathcal{N}(\omega).
\end{align*}
If $j_0$ is the unique index such that $\omega\in V_{j_0}$ and we define the set
\begin{align*}
 \overline{B}_\ell\coloneqq \{y\in \overline{B}_\ell(y_0)\mid \dist_Y(y, \partial Y)\geq 2\ell^{-1}\},
\end{align*}
for $\ell\in \N$, by~\eqref{eqn: double transform} we can then apply Lemma~\ref{lem: avg c_p-convex} with $f_m=\xi_{k, m}(\Xi_{j_0}(\omega, \cdot))$ 
and $\lambda_{\ell, m'}=\zeta_{k, m}(\omega)$ independent of $\ell'\in\mathbb{N}$
(since the sequence $(\zeta^{\avg}_{k, M}(\omega))_{M\in \N}$ converges, it is also uniformly bounded)
for $1\leq k\leq K-1$ to obtain that $(\tilde{\eta}_{k, M_{\ell'}}^{\avg}(\omega, \cdot))_{\ell'\in \N}$ is uniformly Lipschitz 
on $\overline{B}_\ell$ for each $\ell\in \N$.
By~\eqref{eqn: partition version zero at y0} we see $\tilde{\eta}_{k, M_{\ell'}}^{\avg}(\omega, y_0)=0$ for all $k$, thus $(\tilde{\eta}_{k, M_{\ell'}}^{\avg}(\omega, \cdot))_{\ell'\in \N}$ is also bounded on~$\overline{B}_\ell$ and we may apply the Arzel{\`a}--Ascoli theorem to obtain a subsequence of $\tilde{\eta}_{k, M_{\ell'}}^{\avg}(\omega, \cdot)$ that converges uniformly on $\overline{B}_\ell$, necessarily to $\tilde{\eta}_k(\omega, \cdot)$.
By another diagonalization argument, this implies there is a continuous extension of $\tilde{\eta}_k(\omega, \cdot)$ to all of $Y$ for each $\omega\in \Omega'$; we continue to denote this extension by $\tilde{\eta}_k(\omega, \cdot)$. Since we had started with an \emph{arbitrary} increasing sequence $(M_{\ell'})_{\ell'\in \N}$, we conclude that (for the full original sequence) $\tilde{\eta}^{\avg}_{k, M}(\omega, t)$ converges to $\tilde{\eta}_k(\omega, t)$ as $M\to\infty$
 for any fixed $\omega\in \Omega'$, and this convergence is uniform in $t$ when restricted to~$\overline{B}_\ell$ for any $\ell\in \N$.  
By~\eqref{eqn: eta_k sum to zero} we have 
\[
\sum_{k=1}^K \tilde{\eta}_{k, M}^{\avg}\equiv 0,
\] 
hence we see the same uniform convergence claim holds for $(\tilde{\eta}_{K, M}^{\avg}(\omega, \cdot))_{M\in \N}$ as well. Finally by disjointness of the $V_j$,
\begin{align*}
\sum_{j\in \N}\mathds{1}_{V_j}(\pi(v))\tilde{\eta}^{\avg}_{k, M}(\Xi_j^{-1}(v))
&=\sum_{j\in \N}\mathds{1}_{V_j}(\pi(v))\left( \frac{1}{M}\sum_{m=1}^M\sum_{j'\in \N}\mathds{1}_{V_{j'}}(\pi(v))\eta_{k, m}(\Xi_{j'}(\Xi_j^{-1}(v)))\right)\\
&= \frac{1}{M}\sum_{m=1}^M\sum_{j\in \N}\mathds{1}_{V_j}(\pi(v))\eta_{k, m}(v)
=\frac{1}{M}\sum_{m=1}^M\eta_{k, m}(v)=\eta_{k, M}^{\avg}(v),
\end{align*}
hence defining 
\[
\eta_k(v)\coloneqq \mathds{1}_{\pi^{-1}(\Omega')}(v)\cdot \sum_{j\in \N}\mathds{1}_{V_j}(\pi(v))\tilde{\eta}_k(\Xi_j^{-1}(v))
\]
we see $\eta_k$ is Borel, satisfies~\eqref{eqn: etas converge everywhere}, and the uniform convergence claim for $(\eta_{k, M}^{\avg})_{M\in \N}$ holds. Finally, this uniform convergence implies $\eta_k$ is continuous when restricted to $\pi^{-1}(\{\omega\})$ for any $\omega\in \Omega'$.
\end{proof}
For a final lemma, we prove measurability properties of certain functions constructed from the limiting functions obtained by Lemma~\ref{lem: subsequence of maximizing sequence}.
\begin{lemma}\label{lem: measurability of limits}
Under the same hypotheses and notation as Lemma~\ref{lem: subsequence of maximizing sequence}, for $1\leq k\leq K$, define (with the convention $0/0=0$)
\begin{align*}
\Omega_k\coloneqq \{\omega\in \Omega'\mid \zeta_k(\omega)\neq 0\},\qquad
    \xi_k(v)\coloneqq \frac{\eta_k(v)}{\zeta_k(\pi(v))}\mathds{1}_{\Omega_k}(\pi(v))\quad \text{for } v\in E.
\end{align*}
Then for any $\varepsilon\in (0, \sigma(\Omega_k))$ there exists a Borel set $\Omega_{k, \varepsilon}\subset \Omega\setminus \Omega_k$ with $\sigma(\Omega_{k, \varepsilon})<\varepsilon$ such that $\zeta_{k, M}^{\avg}$ converges uniformly to zero on $\Omega\setminus (\Omega_k\cup \Omega_{k, \varepsilon})$, and for any $\mathfrak{n}\in \mathcal{P}^\sigma_{p, q}(E)$, the functions defined on $\Omega$ by
\begin{align}
    \omega&\mapsto 
         -\mathds{1}_{\Omega'}(\omega)\int_E \eta_kd\mathfrak{n}^\omega,
    \label{eqn: potential measurable}\\
    \omega&\mapsto 
    \left[-\zeta_k(\omega)\mathds{1}_{\Omega_k}(\omega)\int_E S_{\lambda_k}\xi_kd\mathfrak{m}^\omega_k+\mathds{1}_{\Omega\setminus(\Omega_k\cup \Omega_{k, \varepsilon})}(\omega)\inf_{\pi^{-1}(\{\omega\})}\eta_k\right]\mathds{1}_{\Omega'}(\omega)
        \label{eqn: transform measurable}
\end{align}
are $\mathcal{B}_\sigma$-measurable.
\end{lemma}
\begin{proof}
Fix $1\leq k\leq K$. For any $\varepsilon>0$, by Egorov's theorem there is a Borel set $\Omega_{k, \varepsilon}\subset \Omega\setminus \Omega_k$ with $\sigma(\Omega_{k, \varepsilon})<\varepsilon$ such that $\zeta_{k, M}^{\avg}$ converges uniformly to zero on $\Omega\setminus (\Omega_k\cup \Omega_{k, \varepsilon})$. 

We begin with the measurability of~\eqref{eqn: potential measurable}. Since $\eta_k$ is Borel, hence by~\nameref{thm: disintegration} the integral of its positive and negative parts respectively against $\mathfrak{n}^\omega$ are Borel in $\omega$. Thus to obtain measurability of~\eqref{eqn: potential measurable}, it is sufficient to show the integral is finite from below for $\sigma$-a.e. $\omega\in\Omega$. 
To this end, for each $1\leq k\leq K$, $\omega\in \Omega'$, and $u$, $v\in \pi^{-1}(\{\omega\})$, calculating as in~\eqref{etasd} we must have
\begin{align*}
\eta_k(v)
&=\lim_{M\to\infty}\frac1M\sum_{m=1}^M\eta_{k, m}(v)\\
&\geq 
\limsup_{M\to\infty}
\frac1M\sum_{m=1}^M
\Big[-\zeta_{k, m}(\omega)S_{\lambda_k}\xi_{k, m}(u)-2^{p-1}\zeta_{k, m}(\omega)
\left(\dist_{E,y_0}^p(\omega,u)+\dist_{E,y_0}^p(\omega,v)\right)
\Big]\\
&\geq 
\limsup_{M\to\infty}
\left(-\frac{1}{M}\sum_{m=1}^M\zeta_{k, m}(\omega)S_{\lambda_k}\xi_{k, m}(u)\right)
-2^{p-1}\zeta_{k}(\omega)
\left(\dist_{E,y_0}^p(\omega,u)+\dist_{E,y_0}^p(\omega,v)\right).
\end{align*}
As $\mathfrak{m}_k^\omega$ and $\mathfrak{n}^\omega$ are supported on $\pi^{-1}(\{\omega\})$, integrating against $(\mathfrak{m}_k^\omega\otimes \mathfrak{n}^\omega)(u,v)$ and using~\eqref{eqn: disint moment estimate} yields
\begin{align}
\begin{split}
\int_E\eta_kd\mathfrak{n}^\omega
&\geq \int_E\limsup_{M\to \infty}
\left( -\frac{1}{M}\sum_{m=1}^M \zeta_{k, m}(\omega)S_{\lambda_k}\xi_{k, m}\right)d\mathfrak{m}_k^\omega\\
&\quad-2^{2p-2}\zeta_k(\omega)\left(2\widetilde{C}+\mk_p^E(\delta^\omega_{E,y_0}, \mathfrak{m}_k^\omega)^p+\mk_p^E(\delta^\omega_{E,y_0}, \mathfrak{n}^\omega)^p\right)\\
&\geq \limsup_{M\to \infty}\int_E\left(-\frac{1}{M}\sum_{m=1}^M \zeta_{k, m}(\omega)S_{\lambda_k}\xi_{k, m}\right)d\mathfrak{m}_k^\omega\\
&\quad-2^{2p-2}\zeta_k(\omega)\left(2\widetilde{C}+\mk_p^E(\delta^\omega_{E,y_0}, \mathfrak{m}_k^\omega)^p+\mk_p^E(\delta^\omega_{E,y_0}, \mathfrak{n}^\omega)^p\right);\label{eqn: integral average bound eta}
\end{split}
\end{align}
here we are able to apply Fatou's lemma to obtain the final inequality due to the fact that by~\eqref{eqn: transform y0 bound}, we have
\begin{align*}
    -\frac{1}{M}\sum_{m=1}^M \zeta_{k, m}(\omega)S_{\lambda_k}\xi_{k, m}(u)
    &\leq \left( \sup_{M'\in \N}\zeta^{\avg}_{k, M'}(\omega)\right)\cdot \lambda_k \dist_{E,y_0}^p(\pi(u),u),
\end{align*}
where the expression on the right belongs to $L^1(\mathfrak{m}_k^\omega)$ for $\sigma$-a.e. $\omega$ by~\eqref{eqn: disint moment estimate} combined with  $\mathfrak{m}_k\in \mathcal{P}^\sigma_{p, q}(E)$. 
Also using~\eqref{eqn: disint moment estimate},
\begin{align*}
    \int_E\left(-\frac{1}{M}\sum_{m=1}^M\zeta_{k, m}(\omega)S_{\lambda_k}\xi_{k, m}\right)d\mathfrak{m}_k^\omega
    &\leq 2^{p-1}\lambda_k\left(\sup_{M'\in \N}\zeta^{\avg}_{k, M'}(\omega)\right)\left(\widetilde{C}+\mk_p^E(\delta^\omega_{E,y_0}, \mathfrak{m}_k^\omega)^p\right)
\end{align*}
and the expression on the right belongs to $L^1(\sigma)$, again due to the fact that $\mathfrak{m}_k\in \mathcal{P}^\sigma_{p, q}(E)$, thus we may integrate the last expression in~\eqref{eqn: integral average bound eta} against $\sigma$ and apply Fatou's lemma and H\"older's inequality to obtain
\begin{align}
\begin{split}
\label{eqn: eta full integral lower bound}
    &\int_\Omega\left[\limsup_{M\to \infty}\int_E\left(-\frac{1}{M}\sum_{m=1}^M
(\zeta_{k, m}\circ\pi)S_{\lambda_k}\xi_{k, m}
\right)d\mathfrak{m}_k^\bullet d\sigma\right]\\
&\quad-2^{2p-2}\int_\Omega\left[\zeta_k \cdot \left(2\widetilde{C}+\mk_p^E(\delta^\bullet_{E,y_0}, \mathfrak{m}_k^\bullet)^p+\mk_p^E(\delta^\bullet_{E,y_0}, \mathfrak{n}^\bullet)^p\right)\right]d\sigma\\
&\geq \limsup_{M\to \infty}\left(-\frac{1}{M}\sum_{m=1}^M\int_\Omega\zeta_{k, m}\int_E S_{\lambda_k}\xi_{k, m}d\mathfrak{m}_k^\bullet d\sigma\right)\\
&\quad-2^{2p-2}\left\|\zeta_k\right\|_{L^{r'}(\sigma)}\left(2\widetilde{C}+\scrmk_{p, q}(\delta^\bullet_{E,y_0}\otimes\sigma, \mathfrak{m}_k^\omega)^p+\scrmk_{p,q}(\delta^\bullet_{E,y_0}\otimes\sigma, \mathfrak{n}^\omega)^p\right)\\
&>-\infty,
\end{split}
\end{align}
where the finiteness follows as in~\eqref{eqn: single term lower bound} with the fact that $\mathfrak{n}$, $\mathfrak{m}_k\in \mathcal{P}^\sigma_{p, q}(E)$.
Hence 
\[
\int_E\eta_kd\mathfrak{n}^\bullet
\]
has a finite lower bound for $\sigma$-a.e. for each $1\leq k\leq K$, yielding the $\mathcal{B}_\sigma$-measurability of~\eqref{eqn: potential measurable}.

Next we show the measurability of \eqref{eqn: transform measurable}. 
Since $Y$ is separable and the function $\eta_k\circ\Xi_{j, \omega}$ is continuous on~$Y$ for each $\omega\in U_j$, 
there exists a countable subset $D$ of $Y$ (independent of $\omega$) such that 
\[
\inf_{v\in \pi^{-1}(\{\omega\})}\eta_k(v)=\inf_{t\in Y}\eta_k(\Xi_{j, \omega}(t))
=
\inf_{t\in D}\eta_k(\Xi_{j, \omega}(t)),
\]
hence the function
\[
\omega \mapsto 
\mathds{1}_{\Omega'}(\omega)\inf_{\pi^{-1}(\{\omega\})}\eta_k
\] 
is $\mathcal{B}_\sigma$-measurable in $\omega$. Again since $S_{\lambda_k}\xi_k$ is Borel, it suffices by~\nameref{thm: disintegration} this time to show that
\[
-\int_E S_{\lambda_k}\xi_kd\mathfrak{m}_k^\omega
<\infty
\quad
\text{for $\sigma$-a.e. $\omega$.}
\]
This follows as by a calculation analogous to~\eqref{eqn: transform y0 bound} applied to $\xi_k$ in place of $\xi_{k, m}$, followed by~\eqref{eqn: disint moment estimate}, we have
\begin{align*}
 -\int_E S_{\lambda_k}\xi_kd\mathfrak{m}_k^\omega
 \leq  \lambda_k\int_E\dist_{E,y_0}^p(\pi(u),u)d\mathfrak{m}_k^\omega(u)
 \leq \lambda_k2^{p-1}(\widetilde{C}+\mk^E_p(\delta^\omega_{E,y_0}, \mathfrak{m}_k^\omega)^p),
\end{align*}
and the last expression is finite for $\sigma$-a.e. $\omega$ as $\mathfrak{m}_k\in \mathcal{P}^\sigma_{p, q}(E)$. 
Thus we have the $\mathcal{B}_\sigma$-measurability of~\eqref{eqn: transform measurable} for $1\leq k\leq K$ as claimed.
\end{proof}
We are now ready to prove uniqueness of $\scrmk_{p, q}$-barycenters under appropriate conditions.
\begin{proof}[Proof of Theorem~\ref{thm: barycenters}~\eqref{thm: disint bary unique}]
By Theorem~\ref{thm: barycenters}~\eqref{thm: disint bary duality}, for $1\leq k \leq K$ and $m\in \mathbb{N}$,
we can take elements $(\zeta_{k, m}, \hat{\xi}_{k, m})_{k=1}^K \in (\mathcal{Z}_{r', \sigma}\times \mathcal{X}_p)^K$
which satisfy 
\begin{align*}
\sum_{k=1}^K (\zeta_{k, m}\circ \pi) \hat{\xi}_{k, m}&=0,\\
-\sum_{k=1}^K\int_\Omega \zeta_{k, m}
 \left(\int_E S_{\lambda_k} \hat{\xi}_{k, m} d \mathfrak{m}_k^\bullet\right) d\sigma
&\leq
 -\sum_{k=1}^K\int_\Omega  \zeta_{k, m+1}
 \left(\int_E S_{\lambda_k} \hat{\xi}_{k, m+1} d \mathfrak{m}_k^\bullet\right) d\sigma\\
&\xrightarrow{m\to\infty}\inf_{\mathfrak{n}\in\mathcal{P}^\sigma_{p, q}(E)}\sum_{k=1}^K \lambda_k \scrmk_{p,q}\left(\mathfrak{m}_k,\mathfrak{n}\right)^p,
\end{align*} 
where this limit is also the value of the supremum for the dual problem in Theorem~\ref{thm: barycenters}~\eqref{thm: disint bary duality}. 
Define
\begin{align*}
    \tilde{\xi}_{k, m}\coloneqq 
    \begin{dcases}
     S_{\lambda_k}(S_{\lambda_k}\hat{\xi}_{k, m}),&\text{if } 1\leq k\leq K-1,\\
     -\frac{1}{(\zeta_{K, m}\circ\pi)}\sum_{i=1}^{K-1}(\zeta_{i, m}\circ\pi)\tilde{\xi}_{i, m},&\text{if }k=K,
    \end{dcases}
\end{align*}
then 
\begin{align}\label{eqn: tilde xi sum to zero}
\sum_{k=1}^K
(\zeta_{k, m}\circ\pi)\tilde{\xi}_{k, m}\equiv 0.
\end{align}
For $1\leq k\leq K-1$, it is classical that
\begin{align}
S_{\lambda_k}\tilde{\xi}_{k, m}&=S_{\lambda_k}(S_{\lambda_k}(S_{\lambda_k}\hat{\xi}_{k, m}))=S_{\lambda_k}\hat{\xi}_{k, m},\notag\\ 
\hat{\xi}_{k, m}&\geq\tilde{\xi}_{k, m}\geq -S_{\lambda_k} \hat{\xi}_{k, m}.\label{eqn: double transform bound}
\end{align}
This yields
\begin{align*}
    \tilde{\xi}_{K, m}
    =-\frac{1}{(\zeta_{K, m}\circ \pi)}\sum_{k=1}^{K-1}(\zeta_{k, m}\circ \pi)\tilde{\xi}_{k, m}
    \geq -\frac{1}{(\zeta_{K, m}\circ \pi)}\sum_{k=1}^{K-1}(\zeta_{k, m}\circ \pi)\hat{\xi}_{k, m}=\hat{\xi}_{K, m},
\end{align*}
hence $-S_{\lambda_K}\tilde{\xi}_{K, m}\geq -S_{\lambda_K}\hat{\xi}_{K, m}$.  
For $1\leq k\leq K-1$, since 
\eqref{eqn: double transform bound} holds 
and $\hat{\xi}_{k, m}\in \mathcal{X}_p$, by~\cite{KitagawaTakatsu25a}*{Lemma 2.18} we see $\tilde{\xi}_{k, m}$ is bounded on bounded subsets of $\pi^{-1}(\{\omega\})$ when $\omega\in \Omega$ is fixed. 
Thus composing with $\Xi_{j, \omega}$ for some appropriate $j$, by Corollary~\ref{cor: bounded is lipschitz}, we have that $\tilde{\xi}_{k, m}\vert_{\pi^{-1}(\{\omega\})}$ is continuous for all $1\leq k\leq K-1$ and $\omega\in \Omega$, this also implies $\tilde{\xi}_{K, m}\vert_{\pi^{-1}(\{\omega\})}$ is also continuous. Finally, by definition of $\{V_j\}_{j\in \N}$ we see 
\begin{align*}
\sum_{j\in \N}\chi_j\mathds{1}_{V_j}>0\quad\text{on }\Omega,
\end{align*}
thus for $1\leq k\leq K$ and $v\in E$, we can define 
\begin{align*}
    \xi_{k, m}(v)&\coloneqq
     \tilde{\xi}_{k, m}(v)-\frac{1}{\displaystyle\sum_{j'\in \N}\chi_{j'}(\pi(v))\mathds{1}_{V_{j'}}(\pi(v))}{\displaystyle\sum_{j\in \N}\chi_j(\pi(v))\mathds{1}_{V_j}(\pi(v))\tilde{\xi}_{k, m}(\Xi_j(\pi(v), y_0))},\\
\eta_{k, m}(v)&\coloneqq\zeta_{k, m}(\pi(v))\xi_{k, m}(v),
\end{align*}
then 
\begin{align*}
\sum_{j\in \N}\chi_j(\omega)\mathds{1}_{V_j}(\omega)\xi_{k, m}(\Xi_j(\omega, y_0))=\sum_{j\in \N}\chi_j(\omega)\mathds{1}_{V_j}(\omega)\eta_{k, m}(\Xi_j(\omega, y_0))=0
\end{align*}
for all $k$, $m$, and $\omega\in \Omega$ 
and we can calculate
\begin{align*}
S_{\lambda_k}\xi_{k, m}(u)=S_{\lambda_k}\tilde{\xi}_{k, m}(u) +\sum_{j\in \N}\chi_j(\pi(u))\mathds{1}_{V_j}(\pi(u))\tilde{\xi}_{k, m}(\Xi_j(\pi(u), y_0)),\qquad
\sum_{k=1}^K \eta_{k, m} \equiv 0,
\end{align*}
for all $m$. 
Since (using~\eqref{eqn: tilde xi sum to zero} to obtain the last line below)
\begin{align*}
    &-\sum_{k=1}^K\int_\Omega\zeta_{k, m}
 \left(\int_E S_{\lambda_k} \xi_{k, m} d \mathfrak{m}^\bullet\right)d\sigma\\
& =-\sum_{k=1}^K\int_\Omega\zeta_{k, m}
 \int_E \left(S_{\lambda_k} \tilde{\xi}_{k, m}+\frac{1}{\sum_{j'\in \N}\chi_{j'}\mathds{1}_{V_{j'}}}{\sum_{j\in \N}\chi_j\mathds{1}_{V_j}\tilde{\xi}_{k, m}(\Xi_j(\cdot, y_0))} \right)d \mathfrak{m}^\bullet d\sigma\\
 &=-\sum_{k=1}^K\int_\Omega\zeta_{k, m}
 \int_E S_{\lambda_k} \tilde{\xi}_{k, m} d \mathfrak{m}^\bullet d\sigma,
\end{align*}
we see that
\begin{align}\label{eqn: max sequence}
    \limsup_{m\to\infty}\left(-\sum_{k=1}^K\int_\Omega  \zeta_{k, m}
 \left(\int_E S_{\lambda_k} \xi_{k, m} d \mathfrak{m}_k^\bullet\right) d\sigma\right)
\geq\inf_{\mathcal{P}^\sigma_{p, q}(E)}\sum_{k=1}^K \lambda_k \scrmk_{p,q}\left(\mathfrak{m}_k,\cdot\right)^p.
\end{align}
Thus we may pass to a subsequence to assume
\begin{align*}
     \inf_{m\in\N}\left(-\sum_{k=1}^K\int_\Omega\zeta_{k, m}(\omega)\int_E S_{\lambda_k}\xi_{k, m} d\mathfrak{m}^\omega_kd\sigma(\omega)\right)
     &\geq \inf_{\mathcal{P}^\sigma_{p, q}(E)}\sum_{k=1}^K \lambda_k \scrmk_{p,q}\left(\mathfrak{m}_k,\cdot\right)^p-1>-\infty.
\end{align*}
If $p<q$, then we have $1<r'<\infty$ hence $L^{r'}(\sigma)$ is reflexive. Since $(\zeta_{k, m})_{m\in\N}$ is bounded in $L^{r'}(\sigma)$ for each $1\leq k\leq K$, we can pass to a subsequence which can be assumed to converge weakly in $L^{r'}(\sigma)$ to some $\zeta_k$. If $p=q$, then by Remark~\ref{rem: p=q duality} we may assume that each $\zeta_{k, m}\equiv 1$.
Thus we may apply Lemmas~\ref{lem: subsequence of maximizing sequence} and~\ref{lem: measurability of limits} to $(\xi_{k, m})_{m\in \N}$ and $(\zeta_{k, m})_{m\in\N}$; let $\eta_k$, $\xi_k$, and $\Omega_k$ be obtained from applying these Lemmas; we also continue using the notation $\eta^{\avg}_{k, M}$ and $\zeta^{\avg}_{k, M}$.

Now suppose $\mathfrak{n}\in \mathcal{P}^\sigma_{p, q}(E)$ is a minimizer of $\sum_{k=1}^K \lambda_k \scrmk_{p,q}\left(\mathfrak{m}_k,\cdot\right)^p$, and for $1\leq k\leq K$, $j\in \N$ let~$\Omega_{k, j}$ be the set obtained from Lemma~\ref{lem: measurability of limits} with $\varepsilon=j^{-1}\sigma(\Omega_k)$ if $\sigma(\Omega_k)>0$, and the empty set otherwise. If we denote 
\begin{align*}
    \xi^{\avg}_{k, M}\coloneqq \frac{\eta^{\avg}_{k, M}}{(\zeta^{\avg}_{k, M}\circ \pi)},
\end{align*}
then since $\xi^{\avg}_{k, M}(v)\to \xi_k(v)$ as $M\to \infty$ whenever $\pi(v)\in \Omega_k$, for all $\omega\in \Omega_k$ and $u\in \pi^{-1}(\{\omega\})$ we have
\begin{align}
\begin{split}
    \limsup_{M\to\infty}\left( -\zeta^{\avg}_{k, M}(\omega) S_{\lambda_k}\xi^{\avg}_{k, M}(u)\right)
    &=\limsup_{M\to\infty}\left[\zeta^{\avg}_{k, M}(\omega) 
     \inf_{v\in \pi^{-1}(\{\pi(u)\})}\left( \lambda_k \dist_E(u, v)^p+\xi^{\avg}_{k, M}(v)\right)\right]\\
    &\leq \inf_{v\in \pi^{-1}(\{\pi(u)\})}\limsup_{M\to\infty}[\zeta^{\avg}_{k, M}(\omega) (\lambda_k \dist_E(u, v)^p+\xi^{\avg}_{k, M}(v))]\\
    &=-\zeta_k(\omega) S_{\lambda_k}\xi_k(u),\label{eqn: limsup estimate}
\end{split}
\end{align} 
where we use that 
\[
\limsup_{\ell\to\infty}(a_\ell b_\ell)=\left(\lim_{\ell\to\infty}a_\ell\right)\left(\limsup_{\ell\to\infty}b_\ell\right)
\]
for any sequences $(a_\ell)_{\ell\in\mathbb{N}},(b_\ell)_{\ell\in\mathbb{N}}$ 
such that 
$(a_\ell)_{\ell\in\mathbb{N}}$ converges to a positive number.
Meanwhile for $\omega\in \Omega'\setminus \Omega_k$ and $u\in \pi^{-1}(\{\omega\})$ we have
\begin{align}
\begin{split}
    \limsup_{M\to\infty}\left( -\zeta^{\avg}_{k, M}(\omega) S_{\lambda_k}\xi^{\avg}_{k, M}(u)\right)
    &\leq \inf_{v\in \pi^{-1}(\{\pi(u)\})}\limsup_{M\to\infty}\left(\lambda_k \zeta^{\avg}_{k, M}(\omega)\dist_E(u, v)^p+\eta^{\avg}_{k, M}(v)\right)\\
&= \inf_{v\in \pi^{-1}(\{\pi(u)\})}\eta_k(v).\label{eqn: zeta_k=0 limsup estimate}
\end{split}
\end{align}
Since $\zeta_{k, M}^{\avg}$ converges $\sigma$-a.e., it is bounded $\sigma$-a.e, and by~\eqref{eqn: transform y0 bound},
\begin{align*}
    -\zeta_{k, M}^{\avg}(\omega)S_{\lambda_k}\xi^{\avg}_{k, M}(u)\leq \left(\sup_{M'\in \N}\zeta^{\avg}_{k, M'}(\omega)\right)\cdot \lambda_k\dist_{E,y_0}^p(\pi(u),u)
\end{align*}
for $\sigma$-a.e. $\omega$. Again since $\mathfrak{m}_k\in \mathcal{P}^\sigma_{p, q}(E)$, by~\eqref{eqn: disint moment estimate} we have
\begin{align}\label{eqn: distance Lr}
\int_E \dist_{E,y_0}^p(\pi(u),u) d\mathfrak{m}_k^\bullet(u)\in L^r(\sigma)\subset L^1(\sigma),    
\end{align} hence we may use Fatou's lemma to obtain 
\begin{align}\label{eqn: fatou 1}
\begin{split}
\limsup_{M\to\infty}\int_E
\left( -\zeta_{k, M}^{\avg}(\omega) S_{\lambda_k}\xi^{\avg}_{k, M} \right)d\mathfrak{m}_k^\omega
&\leq \int_E\limsup_{M\to\infty}\left(
-\zeta_{k, M}^{\avg}(\omega) S_{\lambda_k}\xi^{\avg}_{k, M} \right)d\mathfrak{m}_k^\omega
\end{split}
\end{align}
for $\sigma$-a.e. $\omega$.
Since $\sigma$ has finite total measure, $L^{r'}(\sigma)$-convergence of the $\zeta_{k, M}^{\avg}$ implies the restricted sequence $(\zeta_{k, M}^{\avg}\mathds{1}_{\Omega_k})_{M\in \N}$ converges in~$L^1(\sigma)$, necessarily to $\zeta_k\mathds{1}_{\Omega_k}=\zeta_k$. 

Next suppose $\left\|\zeta_k\right\|_{L^1(\sigma)}>0$, 
then we have $\left\|\zeta_{k, M}^{\avg}\mathds{1}_{\Omega_k}\right\|_{L^1(\sigma)}>0$ for all $M$ sufficiently large, and 
\begin{align*}
    \left\|\zeta_{k, M}^{\avg}\mathds{1}_{\Omega_k}\right\|_{L^1(\sigma)}^{-1}\int_{\Omega'}\zeta_{k, M}^{\avg}\mathds{1}_{\Omega_k}d\sigma\xrightarrow{M\to\infty}\left\|\zeta_k\right\|_{L^1(\sigma)}^{-1}\int_{\Omega'}\zeta_kd\sigma
\end{align*}
for any $\Omega'\in \mathcal{B}_\sigma$. Thus we can view 
\[
\left(\left\|\zeta_{k, M}^{\avg}\mathds{1}_{\Omega_k}\right\|_{L^1(\sigma)}^{-1}\zeta_{k, M}^{\avg}\mathds{1}_{\Omega_k}\sigma\right)_{M\in \N}
\] as a sequence 
in $\mathcal{P}(\Omega)$ that converges setwise to the probability measure $\left\|\zeta_k\right\|_{L^1(\sigma)}^{-1}\zeta_k\sigma$. 
Thus by~\eqref{eqn: distance Lr} and using~\eqref{eqn: disint moment estimate}, recalling the $L^1(\sigma)$- and $L^{r'}(\sigma)$-convergence of $(\zeta_{k, M}^{\avg}\mathds{1}_{\Omega_k})_{M\in\N}$ to $\zeta_k$ yields
\begin{align*}
   & \limsup_{M\to\infty}
     \int_{\Omega_k} \frac{\zeta_{k, M}^{\avg}}{\left\|\zeta_{k, M}^{\avg}\mathds{1}_{\Omega_k}
     \right\|_{L^1(\sigma)}}\int_E \dist_{E,y_0}^p(\pi(u),u)  d\mathfrak{m}_k^\bullet(u) d\sigma\\
    &=\frac{1}{\left\|\zeta_k\right\|_{L^1(\sigma)}}
      \limsup_{M\to\infty}\int_{\Omega_k} \zeta_{k, M}^{\avg}\int_E \dist_{E,y_0}^p(\pi(u),u) d\mathfrak{m}_k^\bullet(u) d\sigma\\
      &=\frac{1}{\left\|\zeta_k\right\|_{L^1(\sigma)}}
      \int_{\Omega_k} \zeta_k\int_E \dist_{E,y_0}^p(\pi(u),u) d\mathfrak{m}_k^\bullet(u) d\sigma\\
    &\leq\frac{2^{p-1}}{\left\|\zeta_k\right\|_{L^1(\sigma)}}
      \int_\Omega \zeta_k \left(\widetilde{C}+\mk_p^E(\delta^\bullet_{E,y_0}, \mathfrak{m}_k^\bullet)^p\right) d\sigma\\
   &\leq\frac{2^{p-1}}{\left\|\zeta_k\right\|_{L^1(\sigma)}}\left\| \zeta_k\right\|_{L^{r'}(\sigma)}\cdot \left(\widetilde{C}+\scrmk_{p,q}(\delta^\bullet_{E,y_0}\otimes\sigma, \mathfrak{m}_k)^p\right)
   <\infty.
\end{align*}
Since
\[
-\int_E S_{\lambda_k}\xi^{\avg}_{k, M} d\mathfrak{m}_k^\omega\leq \lambda_k \int_E \dist_{E,y_0}^p(\pi(u),u)d\mathfrak{m}_k^\omega(u)
\]
we may apply Fatou's lemma for sequences of probability measures,~\cite{FeinbergKasyanovZadoianchuk14}*{Theorem 4.1}, with the choices 
\begin{align*}
\mu_n&=\frac{\zeta_{k, n}^{\avg}\mathds{1}_{\Omega_k}\sigma}{\left\|\zeta_{k, n}^{\avg}\mathds{1}_{\Omega_k}\right\|_{L^1(\sigma)}},
\qquad
g_n=-\lambda_k \int_E \dist_{E,y_0}^p(\pi(u),u) d\mathfrak{m}_k^\bullet(u), \qquad
f_n=\int_E S_{\lambda_k}\xi^{\avg}_{k, n} d\mathfrak{m}_k^\bullet
\end{align*}
in the reference which yields
\begin{align}
\begin{split}
   &\int_{\Omega_k} \limsup_{M\to\infty}\int_E  \left (    -(\zeta_{k, M}^{\avg}\circ\pi) S_{\lambda_k}\xi^{\avg}_{k, M} \right)
   d\mathfrak{m}_k^\bullet d\sigma\\
   &=\left\|\zeta_k\right\|_{L^1(\sigma)}\int_{\Omega_k}\frac{\zeta_k}{\left\|\zeta_k\right\|_{L^1(\sigma)}}\limsup_{M\to\infty}\left(-\int_E S_{\lambda_k}\xi^{\avg}_{k, M} d\mathfrak{m}_k^\bullet \right)d\sigma \\
   &\geq \left\|\zeta_k\right\|_{L^1(\sigma)}\limsup_{M\to\infty}\left(-\int_{\Omega_k}\frac{\zeta_{k, M}^{\avg}}{\left\|\zeta_{k, M}^{\avg}\mathds{1}_{\Omega_k}\right\|_{L^1(\sigma)}}
                 \int_E S_{\lambda_k}\xi^{\avg}_{k, M} d\mathfrak{m}_k^\bullet d\sigma \right)\\
   &= \limsup_{M\to\infty}\left(-\int_{\Omega_k}\zeta_{k, M}^{\avg}\int_E S_{\lambda_k}\xi^{\avg}_{k, M} d\mathfrak{m}_k^\bullet d\sigma\right);\label{eqn: fatou 2}
\end{split}
\end{align}
above we have used that 
\[
\lim_{M\to\infty}\zeta_{k, M}^{\avg}>0
\quad\text{on } \Omega_k.
\]
If $\left\| \zeta_k\right\|_{L^1(\sigma)}=0$, we would have $\sigma(\Omega_k)=0$ and the same inequality~\eqref{eqn: fatou 2} holds. By a calculation analogous to \eqref{eqn: H convexity}, for any $M\in \N$ we have
\begin{align}\label{eqn: lower bound by average}
    -(\zeta_{k, M}^{\avg} \circ\pi)S_{\lambda_k}\xi^{\avg}_{k, M}
    &\geq -\frac{1}{M}\sum_{m=1}^M (\zeta_{k, m} \circ \pi)S_{\lambda_k}\xi_{k, m},
\end{align}
thus combining the above with~\eqref{eqn: fatou 1} and~\eqref{eqn: fatou 2}  we see 
\begin{align}
\begin{split}
\int_{\Omega_k}\int_E\limsup_{M\to\infty}\left(-(\zeta_{k, M}^{\avg} \circ\pi)S_{\lambda_k}\xi^{\avg}_{k, M}\right) d\mathfrak{m}_k^\bullet d\sigma
&\geq \limsup_{M\to\infty}\left(-\frac{1}{M}\sum_{m=1}^M\int_{\Omega_k}\zeta_{k, m}\int_E S_{\lambda_k}\xi_{k, m} d\mathfrak{m}_k^\bullet d\sigma\right). \label{eqn: limit lower bound by average}
\end{split}
\end{align}
Now since $(\zeta_{k, M}^{\avg})_{M\in \mathbb{N}}$ converges uniformly to $0$ on $\Omega\setminus(\Omega_k \cup\Omega_{k, j})$
, for all $M$ sufficiently large we have 
\[
-\zeta_{k, M}^{\avg}(\pi(u)) S_{\lambda_k}\xi^{\avg}_{k, M}(u) \leq \lambda_k\dist_{E,y_0}^p(\pi(u),u)
\quad
\text{for }u\in\pi^{-1}(\Omega\setminus (\Omega_k\cup \Omega_{k, j})).
\]
Since the expression on the right-hand side has finite integral with respect to $\mathfrak{m}_k$, by Fatou's lemma and~\eqref{eqn: lower bound by average} we have
\begin{align*}
    &\int_{\Omega\setminus (\Omega_k\cup \Omega_{k, j})} \int_E  \limsup_{M\to\infty}\left (-(\zeta_{k, M}^{\avg}\circ\pi)  S_{\lambda_k}\xi^{\avg}_{k, M} \right)
   d\mathfrak{m}_k^\bullet d\sigma\\
   &\geq \limsup_{M\to\infty}\int_{\Omega\setminus (\Omega_k\cup \Omega_{k, j})}\int_E \left(-\frac{1}{M}\sum_{m=1}^M(\zeta_{k, m}\circ\pi)S_{\lambda_k}\xi_{k, m} d\mathfrak{m}_k^\bullet d\sigma\right),
\end{align*}
thus combining with~\eqref{eqn: limit lower bound by average} we have
\begin{align}\label{eqn: limit lower bound by average 2}
\begin{split}
    &\int_{\Omega\setminus\Omega_{k, j}}\int_E\limsup_{M\to\infty}\left(-(\zeta_{k, M}^{\avg}\circ\pi) S_{\lambda_k}\xi^{\avg}_{k, M}\right) d\mathfrak{m}_k^\bullet d\sigma\\
&\geq \limsup_{M\to\infty}\int_{\Omega\setminus\Omega_{k, j}}\int_E \left(-\frac{1}{M}\sum_{m=1}^M(\zeta_{k, m}\circ\pi)S_{\lambda_k}\xi_{k, m} \right) d\mathfrak{m}_k^\bullet d\sigma. 
\end{split}
\end{align}
By the $L^{r'}(\sigma)$-convergence of $(\zeta_{k, M}^{\avg})_{M\in \mathbb{N}}$ to $0$ on $\Omega_{k, j}$ and~\eqref{eqn: distance Lr}, we find
\begin{align*}
    &\limsup_{M\to \infty}\int_{\Omega_{k, j}}\int_E\left(-\frac{1}{M}\sum_{m=1}^M(\zeta_{k, m}\circ \pi) S_{\lambda_k}\xi_{k, m}\right)d\mathfrak{m}_k^\bullet d\sigma \\
&\leq \limsup_{M\to \infty}\left\| \zeta^{\avg}_{k, M}\mathds{1}_{\Omega_{k, j}}\right\|_{L^{r'}(\sigma)}\left\|\int_E\lambda_k \sum_{j\in \N}\chi_j(\pi(u))\dist_E(\Xi_{j, \pi(u)}(y_0), u)^pd\mathfrak{m}_k^\bullet(u) \right\|_{L^r(\sigma)}
=0,
\end{align*}
which in turn yields 
\begin{align*}
\begin{split}
    &\limsup_{M\to \infty}\int_{\Omega\setminus\Omega_{k, j}}\int_E\left(-\frac{1}{M}\sum_{m=1}^M(\zeta_{k, m}\circ\pi) S_{\lambda_k}\xi_{k, m}\right)d\mathfrak{m}_k^\bullet d\sigma \\
 &\geq\limsup_{M\to \infty}\int_\Omega\int_E\left(-\frac{1}{M}\sum_{m=1}^M(\zeta_{k, m}\circ\pi) S_{\lambda_k}\xi_{k, m}\right)d\mathfrak{m}_k^\bullet d\sigma \\
 &\quad-\limsup_{M\to \infty}\int_{\Omega_{k, j}}\int_E\left(-\frac{1}{M}\sum_{m=1}^M(\zeta_{k, m}\circ\pi) S_{\lambda_k}\xi_{k, m}\right)d\mathfrak{m}_k^\bullet d\sigma \\
 &\geq\limsup_{M\to \infty}\int_\Omega\int_E\left(-\frac{1}{M}\sum_{m=1}^M(\zeta_{k, m}\circ\pi) S_{\lambda_k}\xi_{k, m}\right)d\mathfrak{m}_k^\bullet d\sigma \\
 &\geq\inf_{\mathcal{P}^\sigma_{p, q}(E)}\sum_{k=1}^K \lambda_k \scrmk_{p,q}\left(\mathfrak{m}_k,\cdot\right)^p=\sum_{k=1}^K \lambda_k \scrmk_{p,q}\left(\mathfrak{m}_k,\mathfrak{n}\right)^p,
\end{split}
\end{align*}
by~\eqref{eqn: max sequence}.
Combining this with~\eqref{eqn: limsup estimate},~\eqref{eqn: zeta_k=0 limsup estimate}, and~\eqref{eqn: limit lower bound by average 2} and since $\Omega_k$ is disjoint with $\Omega_{k, j}$, we obtain
\begin{align*}
    \sum_{k=1}^K \lambda_k \scrmk_{p,q}\left(\mathfrak{m}_k,\mathfrak{n}\right)^p
    &\leq -\sum_{k=1}^K\int_{\Omega_k}
 \zeta_k(\omega)\int_E S_{\lambda_k} \xi_k d \mathfrak{m}_k^\omega d\sigma(\omega)
 +\sum_{k=1}^K\int_{\Omega\setminus(\Omega_k\cup\Omega_{k, j})}
 \inf_{\pi^{-1}(\{\omega\})}\eta_kd\sigma(\omega).
\end{align*}
Although 
the elements do not necessarily belong to $(\mathcal{Z}_{r', \sigma}\times \mathcal{X}_p)^K$, we do have $\zeta_k\in L^{r'}(\sigma)$ with $\left\| \zeta_k\right\|_{L^{r'}(\sigma)}\leq 1$, and $\xi_k$ continuous on $\pi^{-1}(\{\omega\})$ for $\sigma$-a.e. $\omega$. By~\eqref{eqn: etas sum to zero} 
and the measurability of \eqref{eqn: potential measurable} and \eqref{eqn: transform measurable}, we find
\begin{align}\label{eqn: dual maximizer}
\begin{split}
&\sum_{k=1}^K \lambda_k \scrmk_{p,q}\left(\mathfrak{m}_k,\mathfrak{n}\right)^p\\
    &\leq -\sum_{k=1}^K\int_{\Omega_k}
 \zeta_k \int_E S_{\lambda_k} \xi_k d \mathfrak{m}_k^\bullet d\sigma 
 +\sum_{k=1}^K\int_{\Omega\setminus(\Omega_k\cup\Omega_{k, j})}
 \inf_{\pi^{-1}(\{\omega\})}\eta_kd\sigma(\sigma) -\sum_{k=1}^K\int_\Omega\int_E\eta_kd\mathfrak{n}^\bullet d\sigma\\
&=\sum_{k=1}^K\int_{\Omega_k}
 \left(-\zeta_k \int_E S_{\lambda_k} \xi_k d \mathfrak{m}_k^\bullet-\int_E  \eta_k d \mathfrak{n}^\bullet\right) d\sigma 
 -\sum_{k=1}^K\int_{\Omega_{k, j}}
 \int_E\eta_k d\mathfrak{n}^\bullet d\sigma
 \\
 &\quad+\sum_{k=1}^K\int_{\Omega\setminus(\Omega_k\cup\Omega_{k, j})}
 \int_E\left(-\eta_k+\inf_{\pi^{-1}(\{\omega\})}\eta_k\right)d\mathfrak{n}^\bullet d\sigma(\omega) \\
 &\leq -\sum_{k=1}^K\int_{\Omega_k}
 \zeta_k \left(\int_E S_{\lambda_k} \xi_k d \mathfrak{m}_k^\bullet+\int_E  \xi_kd \mathfrak{n}^\bullet \right)d\sigma 
 -\sum_{k=1}^K\int_{\Omega_{k, j}}
 \int_E\eta_kd\mathfrak{n}^\bullet d\sigma \\
 &\xrightarrow{j\to\infty} -\sum_{k=1}^K\int_{\Omega_k}
 \zeta_k \left(\int_E S_{\lambda_k} \xi_k d \mathfrak{m}_k^\bullet+\int_E \xi_k d \mathfrak{n}^\bullet \right)d\sigma,
\end{split}
\end{align}
where the final limit follows because $\sigma(\Omega_{k, j})\to 0$ as $j\to\infty$, and~\eqref{eqn: etas sum to zero} combined with the estimates~\eqref{eqn: integral average bound eta} and~\eqref{eqn: eta full integral lower bound} implies each $\eta_k\in L^1(\mathfrak{n})$. Since
\begin{align}\label{eqn: usual transform bound}
-\zeta_k(\omega)(S_{\lambda_k} \xi_k(u)+\xi_k(v))\leq \lambda_k\zeta_k(\omega)\dist_E(u, v)^p
\end{align}
for all  $\omega\in\Omega'$ and $u$, $v\in \pi^{-1}(\{\omega\})$,~\eqref{eqn: dual maximizer} implies
\begin{align*}
   &-\sum_{k=1}^K\int_{\Omega_k}
 \zeta_k \left(\int_E S_{\lambda_k} \xi_k d \mathfrak{m}_k^\bullet+\int_E  \xi_k d \mathfrak{n}^\bullet \right)d\sigma \\
 &\geq \sum_{k=1}^K \lambda_k \scrmk_{p,q}\left(\mathfrak{m}_k,\mathfrak{n}\right)^p\\
 &= \sum_{k=1}^K\lambda_k\left\|\mk_p^E(\mathfrak{m}_k^\bullet, \mathfrak{n}^\bullet)^p\mathds{1}_{\Omega\setminus\Omega_k}\right\|_{L^r(\sigma)}+\sum_{k=1}^K\lambda_k\left\|\mk_p^E(\mathfrak{m}_k^\bullet, \mathfrak{n}^\bullet)^p\mathds{1}_{\Omega_k}\right\|_{L^r(\sigma)}\\
 &\geq \sum_{k=1}^K\lambda_k\left\|\mk_p^E(\mathfrak{m}_k^\bullet, \mathfrak{n}^\bullet)^p\mathds{1}_{\Omega_k}\right\|_{L^r(\sigma)}
\\
&\geq -\sum_{k=1}^K\int_{\Omega_k}
 \zeta_k \left(\int_E S_{\lambda_k} \xi_k d \mathfrak{m}_k^\bullet+\int_E  \xi_k d \mathfrak{n}^\bullet \right)d\sigma,
\end{align*}
hence for any $1\leq k\leq K$, for $\sigma$-a.e. $\omega\in\Omega\setminus\Omega_k$, we have $\mk_p^E(\mathfrak{m}_k^\omega, \mathfrak{n}^\omega)=0$, in particular  $\mathfrak{m}_k^\omega=\mathfrak{n}^\omega$.

Now the above also implies 
\begin{align*}
    -\sum_{k=1}^K\int_{\Omega_k} \zeta_k
 \left(\int_E S_{\lambda_k} \xi_k  d \mathfrak{m}_k^\bullet+\int_E  \xi_k d \mathfrak{n}^\bullet\right) d\sigma=\sum_{k=1}^K\lambda_k\left\|\mk_p^E(\mathfrak{m}_k^\bullet, \mathfrak{n}^\bullet)^p\mathds{1}_{\Omega_k}\right\|_{L^r(\sigma)},
\end{align*}
then by~\eqref{eqn: usual transform bound}, each term in the sum on the left of the inequality above is less than or equal to each term in the sum on the right, in particular we have termwise equality for each $1\leq k\leq K$. 

Let $k$ be the distinguished index in our hypothesis. 
Then again using the dual characterization of the $L^r(\sigma)$ norm (\cite{Folland99}*{Proposition 6.13}),
\begin{align*}
\begin{split}
-\int_{\Omega_k} \zeta_k
 \left(\int_E S_{\lambda_k} \xi_k  d \mathfrak{m}_k^\bullet+\int_E  \xi_k d \mathfrak{n}^\bullet\right) d\sigma
& =\lambda_k\left\|\mk_p^E(\mathfrak{m}_k^\bullet, \mathfrak{n}^\bullet)^p\mathds{1}_{\Omega_k}\right\|_{L^r(\sigma)}\\
 &\geq \lambda_k\int_{\Omega_k} \zeta_k\mk_p^E(\mathfrak{m}_k^\bullet, \mathfrak{n}^\bullet)^pd\sigma\\
& \geq -\int_{\Omega_k}\zeta_k
 \left(\int_E S_{\lambda_k} \xi_k d \mathfrak{m}_k^\bullet+\int_E  \xi_k d \mathfrak{n}^\bullet\right) d\sigma.
 \end{split}
\end{align*}
In particular, for $\sigma$-a.e. $\omega\in \Omega_k$ we must have 
\begin{align*}
    -\int_E S_{\lambda_k} \xi_k d \mathfrak{m}_k^\omega-\int_E  \xi_k d \mathfrak{n}^\omega=\lambda_k\mk_p^E(\mathfrak{m}_k^\omega, \mathfrak{n}^\omega)^p.
    \end{align*}
Fix $\omega\in\Omega_k$ where this equality holds, with $\omega\in U_j$ for some $j\in \N$ where the measure $(\Xi_{j, \omega})_\sharp\mathfrak{m}_k^\omega$ is absolutely continuous with respect to $\Vol_Y$. Suppose $j_0$ is the unique index such that $\omega\in V_{j_0}$, then if we define $\phi_\omega$, $\psi_\omega: Y\to \R$ and $\mu_\omega$, $\nu_\omega\in\mathcal{P}_p(Y)$ by
\begin{align*}
 \psi_\omega(s)\coloneqq((S_{\lambda_k} \xi_k)\circ\Xi_{j_0, \omega})^{\lambda_k \dist_Y^p}(s),
 \quad \phi_\omega(t)\coloneqq \psi_\omega^{\lambda_k \dist_Y^p}(t),\quad
 \mu_\omega\coloneqq(\Xi^{-1}_{j_0, \omega})_\sharp \mathfrak{m}_k^\omega,
\quad\nu_\omega\coloneqq (\Xi^{-1}_{j_0, \omega})_\sharp \mathfrak{n}^\omega,
\end{align*}
the above implies
\begin{align*}
    -\int_Y \phi_\omega d\mu_\omega-\int_Y \psi_\omega d \nu_\omega=\lambda_k\mk_p^Y(\mu_\omega, \nu_\omega)^p.
\end{align*}
Since $\mu_\omega=g^{j_0}_j(\omega)_\sharp(\Xi^{-1}_{j, \omega})_\sharp \mathfrak{m}_k^\omega$ and $g^{j_0}_j(\omega)$ is an isometry of $Y$, we also see $\mu_\omega$ is absolutely continuous with respect to $\Vol_Y$. 
Let $\gamma^\omega\in \Pi(\mu_\omega, \nu_\omega)$ be a $p$-optimal coupling between $\mu_\omega$ and $\nu_\omega$. Then we obtain 
\begin{align}\label{eqn: dual equality a.e.}
    - \phi_\omega(t)-\psi_\omega(s)=\lambda_k \dist_Y(t, s)^p,\quad \gamma^\omega\text{-a.e. }(t, s).
    \end{align}
   Since  
\[
-\lambda_k \dist_Y(y_0, t)^p-\psi_\omega(y_0)\leq \phi_\omega(t)\leq S_{\lambda_k} \xi_k(\Xi_{j_0, \omega}(t)),
\]
we see $\phi_\omega$ is bounded on any compact subset of $Y$, and since it is a $\lambda_k\dist_Y^p$-transform of some function, by Corollary~\ref{cor: bounded is lipschitz} $\phi_\omega$ is uniformly Lipschitz on any compact subset of $Y\setminus \partial Y$. Thus by Rademacher's theorem $\phi_\omega$ is differentiable $\Vol_Y$-a.e. on $Y$. 
Let $t\in Y\setminus \partial Y$ be a point of differentiability for $\phi_\omega$ such that there exists $s_t\in Y$ satisfying \eqref{eqn: dual equality a.e.}; 
as $\mathfrak{m}^\omega_k$ is absolutely continuous with respect to $\Vol_Y$, the set of such $t$ has full~$\mathfrak{m}^\omega_k$ measure. 
Let us denote by $\langle \cdot, \cdot \rangle_Y$ the Riemannian metric on~$Y$,
and write $\left| \cdot \right|_Y=\langle \cdot, \cdot\rangle_Y^{1/2}$.
If a function $f$ on $Y$ is differentiable at $t\in Y\setminus\partial Y$, then
\[
f(\exp^Y_t( \varepsilon V ))
=f(t)+\varepsilon \langle V, \nabla_Y f(t)\rangle_Y+o(\varepsilon)\quad\text{ as }\varepsilon \to 0
\]
for any unit tangent vector $V$ to $Y$ at $t$, where $\exp^Y$ is the exponential map of $Y$ and $\nabla_Y f$ is the gradient of $f$. 
This with the choice $f=\phi_\omega$ implies 
\begin{align*}
\dist_Y(\exp_t^Y(\varepsilon V), s_t)^p
\geq
-\phi_\omega(\exp_t^Y(\varepsilon V))
-\psi_\omega(s_t)
&= 
-\varepsilon\langle V, \nabla_Y \phi_\omega(t)   \rangle_Y
-\phi_\omega(t)
-\psi_\omega(s_t)+o(\varepsilon)\\
&=-\varepsilon\langle V, \nabla_Y \phi_\omega(t)   \rangle_Y
+\dist_Y(t, s_t)^p+o(\varepsilon)
\quad
\text{as } \varepsilon \to 0.
\end{align*}
Thus the above shows $t'\mapsto \dist_Y(t', s_t)^p$ is subdifferentiable at $t'=t$, while since  
$\dist_Y^p=(\dist_Y^2)^{p/2}$
we see that~\cite{McCann01}*{Proposition 6} implies superdifferentiability when $s_t\neq t$, hence $t'\mapsto \dist_Y(t', s_t)^p$ is differentiable at $t'=t$ if $s_t\neq t$. 
Since $p>1$, when $s_t\neq t$ 
by taking the derivative of~\eqref{eqn: dual equality a.e.} with respect to $t$, after some tedious but routine calculation we obtain that $\nabla_Y\phi_\omega(t)\neq 0$ and
\[
s_t=
\exp_t^Y\left( 
\left| 
\dfrac{\nabla_Y \phi_\omega(t)}{p\lambda_k}   \right|_Y^{\frac{1}{p-1}} 
\dfrac{\nabla_Y\phi_\omega(t)}{\left|\nabla_Y\phi_\omega(t)\right|_Y}
\right),
\]
and if either $\nabla_Y\phi_\omega(t)= 0$ or $\phi_\omega$ is not superdifferentiable at $t$, we have $s_t=t$. 
This shows that there is a $\mu_\omega$-a.e. single valued map $T^\omega$ on $Y$ such that the pair $(t, T^\omega(t))$ satisfy the equality in \eqref{eqn: dual equality a.e.}. Combining with~\cite{GangboMcCann00}*{Lemma 2.4} necessarily we have that $\gamma^\omega=(\Id\times T^\omega)_\sharp\mu_\omega$. The map $T^\omega$ is entirely determined by $\xi_k$, hence so is the right marginal $\nu_\omega$ for $\sigma$-a.e. $\omega\in \Omega_k$. All together this implies $\mathfrak{n}^\omega$ is determined for $\sigma$-a.e. $\omega$ by $\zeta_k$ or $\xi_k$, thus we see the $\scrmk_{p, q}$-barycenter is unique.
\end{proof}
\begin{proof}[Proof of Theorem~\ref{thm: classical barycenters}]
    We can apply Theorem~\ref{thm: barycenters}~\eqref{thm: disint bary exist},~\eqref{thm: disint bary duality}, and~\eqref{thm: disint bary unique} with any value of $q$ and~$\Omega$ a one-point space, and $\sigma$ the associated delta measure and the claims follow immediately. Regarding the duality result, also recall Remark~\ref{rem: p=q duality}.
\end{proof}

\begin{ack} 
JK was supported in part by National Science Foundation grant DMS-2246606.
AT was supported in part by JSPS KAKENHI Grant Number 24K21513.
\end{ack}

\medskip

\bibliography{Wpq.bib}
\bibliographystyle{alpha}
\end{document}